\documentclass{amsart}

\usepackage{amsmath,amssymb,amsfonts,mathtools,mathrsfs}
\usepackage{tikz-cd}
\usepackage{microtype}
\usepackage[colorlinks=true,linkcolor=blue,citecolor=blue,urlcolor=blue]{hyperref}

\newtheorem{theorem}{Theorem}[section]
\newtheorem{lemma}[theorem]{Lemma}
\newtheorem{proposition}[theorem]{Proposition}
\newtheorem{corollary}[theorem]{Corollary}
\newtheorem*{maintheorem}{Theorem}
\theoremstyle{definition}

\newtheorem{example}[theorem]{Example}
\newtheorem{question}[theorem]{Question}
\newtheorem{remark}[theorem]{Remark}

\newcommand{\id}{\operatorname{id}}
\newcommand{\End}{\operatorname{End}}
\newcommand{\Fun}{\operatorname{Fun}}
\newcommand{\Irr}{\operatorname{Irr}}
\newcommand{\Hom}{\operatorname{Hom}}
\newcommand{\Aut}{\operatorname{Aut}}
\newcommand{\Rep}{\operatorname{Rep}}
\newcommand{\Corep}{\operatorname{Corep}}
\newcommand{\Vect}{\operatorname{Vect}}
\newcommand{\Rad}{\operatorname{Rad}}
\newcommand{\rank}{\operatorname{rank}}
\newcommand{\Map}{\operatorname{Map}}
\newcommand{\Ind}{\operatorname{Ind}}
\newcommand{\Os}{\mathscr{O}}
\newcommand{\C}{\mathscr{C}}
\newcommand{\M}{\mathscr{M}}

\newcommand{\V}{\mathscr{V}}

\newcommand{\Ctimes}{\mathbb{C}^{\times}}
\newcommand{\Gm}{\mathbb{G}_m}

\numberwithin{equation}{section}

\begin{document}

\title[Hopf $2$-cocycles for certain reductive groups]
{Hopf $2$-cocycles for certain affine algebraic reductive groups}

\author{Shlomo Gelaki}
\address{Department of Mathematics, Iowa State University, Ames, IA 50011, USA}
\email{gelaki@iastate.edu}

\date{August 27, 2026}


\begin{abstract}
Let $G$ be a reductive affine algebraic group over $\mathbb C$ with toral
identity component $T$ and finite component group $K$. We describe finite
indecomposable semisimple module categories over $\Rep(G)$ using equivariant
data that incorporate the extension class. We prove that every fiber functor
is classical and determine the image of its Hopf $2$-cocycle support in $K$ and its
intersection with $T$. This yields a minimality criterion involving a finite
stabilizer cocycle and an alternating bicharacter on a sublattice of
$\widehat T$ of index at most $|K|$. When $\dim T=1$, every Hopf $2$-cocycle for $G$ has
finite support of order at most $|K|^2$. We compute all gauge classes for the
split and nonsplit extensions with component group $\mathbb{Z}/2\mathbb{Z}$ and inversion
action, and compare the support formulas with the K\"unneth coordinates for
commutative direct products. Finally, for an arbitrary reductive group $G$, a
minimal Hopf $2$-cocycle in a regular deformation of the trivial Hopf $2$-cocycle forces $G$ to be a torus. 
This uses locality and rigidity of algebraic matrix
coefficients on the semisimple quotient of $G$.
\end{abstract}

\maketitle
\begingroup
\footnotesize
\tableofcontents
\endgroup

\section{Introduction}\label{sec:introduction}

Let $G$ be an affine algebraic group over $\mathbb C$, and let $\Rep(G)$
denote the tensor category of finite-dimensional rational representations of
$G$. A Hopf $2$-cocycle on $\Os(G)$ is equivalently a tensor structure on the
usual forgetful functor
\[
 F:\Rep(G)\longrightarrow\Vect.
\]
We call a tensor functor whose underlying linear functor is isomorphic to
$F$ a \textbf{classical fiber functor}.

Hopf $2$-cocycles for finite groups were classified by Movshev \cite{Mo}, and
the nilpotent affine algebraic case was treated in \cite{EG1,EG2,G1}. A
general classification for reductive affine algebraic groups remains open.
One of the motivating questions is the following.

\begin{question}[{\cite[Question 7.2]{EG1}}]\label{q:torus}
If a reductive affine algebraic group admits a minimal Hopf $2$-cocycle,
must its identity component be a torus? \qed
\end{question}

In Sections~\ref{sec:reductive}--\ref{sec:commutative-direct-products}, we
study the class for which the identity component is already a torus. Thus
we consider an exact sequence
\[
 1\longrightarrow T\longrightarrow G\xrightarrow{\pi}K\longrightarrow1,
\]
where $T$ is a complex torus and $K$ is a finite group. We use the standard
equivariantization attached to this extension,
\[
 \Rep(G)\simeq\Rep(T)^K.
\]
The conjugation action determines the tensor autoequivalences of $\Rep(T)$,
while the extension class determines their coherence maps. We also use the
standard description of module categories over an equivariantization and
its orbit--stabilizer description of simple objects
\cite[Proposition 5.4]{ENO}. In particular, an equivariantization of a finite
semisimple category has rank one exactly when the action on its simple
objects is transitive and the projective cocycle on a stabilizer is
nondegenerate. These categorical ingredients are recalled to fix conventions. 
The finite separable algebra used in
Proposition~\ref{prop:equivariantization-modules} justifies their use here
without a finiteness assumption on the set of simple objects of $\Rep(T)$.

Our main calculation determines the support from the equivariant data.
Recall that a Hopf $2$-cocycle for $G$ is gauge equivalent to the inflation of a
minimal Hopf $2$-cocycle on a closed subgroup of $G$, called its support and determined up
to conjugacy \cite[Theorem 3.1]{G1}. We first prove that every fiber functor
on $\Rep(G)$ is classical
(Theorem~\ref{thm:module-and-fiber-classification}). This identifies rank-one
module categories over $\Rep(G)$ with gauge classes of Hopf $2$-cocycles on the fixed
coordinate algebra $\Os(G)$.

More explicitly, the module categories over $\Rep(G)$ are parametrized by admissible data
$\mathcal D=(S,\Psi,L,U,\tau,\mu)$,
where the subgroup $S\subseteq\widehat T$ has finite index, $L\subseteq K$ is a subgroup, and the
remaining entries specify the module associator and the equivariant
structure. For rank-one data $\mathcal D$, let
\[
 \psi_{\mathcal D}(s,t):=\Psi(s,t;S),\qquad
 \Lambda_{\mathcal D}
 :=\Rad\left(\frac{\psi_{\mathcal D}(s,t)}{\psi_{\mathcal D}(t,s)}\right).
\]
The support theorem gives the following geometric information.

\begin{maintheorem}[Theorem~\ref{prop:factorization}]
Let $\mathcal D$ be rank-one admissible data and let $H_{\mathcal D}$ be the
support of the corresponding Hopf $2$-cocycle. After conjugating the support
in $G$, one has
$\pi(H_{\mathcal D})=L$ and 
$H_{\mathcal D}\cap T=\Lambda_{\mathcal D}^{\perp}$.
Thus, the Hopf $2$-cocycle is minimal if and only if
$L=K$ and $\Lambda_{\mathcal D}=\{1\}$.
\end{maintheorem}

The proof of Theorem~\ref{prop:factorization} computes the internal endomorphism algebra of the unique simple
object as a lattice-graded matrix algebra over
$\mathbb C^{\psi_{\mathcal D}}[S]$. Its center has precisely the weights in
$\Lambda_{\mathcal D}$. The center--support theorem of
\cite[Theorem 3.5]{G1} then identifies these weights with the two support
invariants above. Thus the calculation uses the earlier support theorem,
but extracts the support directly from the equivariant module-category data.
The finite stabilizer condition and the lattice nondegeneracy condition
have separate roles: the former gives a fiber functor, and the latter tests
whether its support contains all of $T$.

The parametrization also has concrete numerical consequences. If $d$ is the
dimension of the unique irreducible projective representation of a
stabilizer, then
\[
 [\widehat T:S]d^2=|L|.
\]
In particular, only sublattices of index at most $|K|$ occur. For a fixed
extension, the sublattices and affine orbit actions therefore range over
finite sets. This does not eliminate the remaining cohomological
compatibility equations. However, when $\dim T=1$, the support of every
Hopf $2$-cocycle is finite and has order at most $|K|^2$
(Corollary~\ref{cor:one-dimensional-support}). We use this to
compute all gauge classes for both extensions of $\Gm$ by $\mathbb{Z}/2\mathbb{Z}$ with
inversion action. The split extension has two classes, whereas the nonsplit
extension, the normalizer of the diagonal torus in $\operatorname{SL}_2$,
has only the trivial class. Example~\ref{ex:split-nonsplit} also exhibits the
obstruction directly in the coherence equations.

For a commutative direct product $G=T\times K$, the character group is
$\widehat T\times\widehat K$, and the standard K\"unneth decomposition gives
\[
 H^2(\widehat T\times\widehat K,\Ctimes)
 \cong H^2(\widehat T,\Ctimes)
 \times\Hom(\widehat T\otimes_{\mathbb Z}\widehat K,\Ctimes)
 \times H^2(\widehat K,\Ctimes).
\]
We identify its coordinates with the orbit, stabilizer, and support data.
The mixed component can compensate for degeneracy of the finite block,
but not for degeneracy of the toral block. If
$\lambda:\widehat T\to K$ is the mixed homomorphism and $N_\xi$ is the
support of the finite component, then minimality is equivalent to
\[
 \Rad(R^\psi)=\{1\},\qquad K=\lambda(\widehat T)N_\xi
\]
(Corollary~\ref{cor:kunneth-equivariantization-minimality}). We also give a
formula for the isotropy radical directly in terms of the three K\"unneth
components (Proposition~\ref{prop:explicit-isotropy-radical}).

Section~\ref{sec:torus-question} returns to Question~\ref{q:torus} for a
general reductive affine algebraic group. The case of a gauge class with an invariant
representative follows from the commutative-support theorem of
\cite[Theorem 4.2]{EGinv}. For regular deformations of the trivial Hopf $2$-cocycle,
we prove a stronger connectedness conclusion as well.

\begin{maintheorem}[Theorem~\ref{thm:deformation-torus}]
Let $G$ be a reductive affine algebraic group over $\mathbb C$. If a regular
family of Hopf $2$-cocycles on $\Os(G)$ has a trivial fiber and a minimal
fiber, then $G$ is a torus.
\end{maintheorem}

Nondegeneracy at one fiber persists over the function field and its formal
parameter field. The locality lemma of \cite[Lemma 6.5]{EG1}, together
with the behavior of formal gauge transformations, first forces $G$ to be
connected. We then use the formal triangular normal form recalled in
\cite[proof of Theorem 6.4]{EG1}. Its comparison isomorphism preserves the
central primitive subspace. On the semisimple quotient, the rigidity of
finite-dimensional representations makes its transpose preserve algebraic
matrix coefficients. Only this smaller coefficient space is needed to
show that the projection of the carrier has full algebraic hull. The
obstruction to a nondegenerate Lie-algebra $2$-cocycle on a nonabelian
reductive Lie algebra then gives the result.

\medskip
\noindent\textbf{Acknowledgements.}
This work was supported by Simons Foundation Award 963288.

\section{Preliminaries}\label{sec:preliminaries}

We work over the field $\mathbb C$ of complex numbers. We use the standard
conventions for Hopf algebras and tensor categories, and refer to
\cite{A,EGNO,M,S} for unexplained notions.

\subsection{\texorpdfstring{Hopf $2$-cocycles}{Hopf 2-cocycles}}\label{sec:hopf-cocycles}

Let $A$ be a Hopf algebra over $\mathbb{C}$.  A \emph{Hopf $2$-cocycle} on $A$ is
a convolution-invertible map $J\in(A\otimes A)^*$ satisfying
\begin{align*}
 J(a_1b_1,c)J(a_2,b_2)&=J(a,b_1c_1)J(b_2,c_2),\qquad
 J(a,1)=\varepsilon(a)=J(1,a)
\end{align*}
for all $a,b,c\in A$.

The cocycle $J$ defines a new Hopf algebra $A^J$.  Its underlying coalgebra is $A$, while its
multiplication is
\[
 a\cdot_J b=J^{-1}(a_1,b_1)a_2b_2J(a_3,b_3).
\]
Equivalently, $J$ defines a tensor structure on the forgetful functor
$$F:\Corep(A)\to\Vect.$$
For right $A$-comodules $X$ and $Y$, with coactions
$x\mapsto x_{(0)}\otimes x_{(1)}$ and $y\mapsto y_{(0)}\otimes y_{(1)}$,  
the tensor constraint is
$$J_{X,Y}(x\otimes y)
 =J^{-1}(x_{(1)},y_{(1)})x_{(0)}\otimes y_{(0)}.$$

For $a\in(A^*)^\times$ with $a(1)=1$, define
\[
 J^a(h,k)
 :=a(h_1k_1)J(h_2,k_2)a^{-1}(h_3)a^{-1}(k_3).
\]
Two Hopf $2$-cocycles on $A$ are {\bf gauge equivalent} if they are related in this way.
Gauge-equivalence classes are in bijection with tensor structures on $F$ up to
monoidal natural isomorphism.

\subsection{Cotriangular Hopf algebras}\label{sec:cotriangular}

A cotriangular structure on a Hopf algebra $A$ is a convolution-invertible
map $R\in(A\otimes A)^*$ satisfying
\begin{align*}
 R^{-1}&=R_{21},\qquad
 R(a,bc)=R(a_1,b)R(a_2,c),\\
 R(ab,c)&=R(b,c_1)R(a,c_2),\qquad
 R(a_1,b_1)b_2a_2=a_1b_1R(a_2,b_2)
\end{align*}
for all $a,b,c\in A$. It is {\bf minimal} if $R$ is nondegenerate.  By \cite[Proposition 2.1]{G1}, every
cotriangular Hopf algebra has a unique minimal cotriangular Hopf algebra quotient.

If $J$ is a Hopf $2$-cocycle on a cotriangular Hopf algebra $(A,R)$, then
$R^J:=J_{21}^{-1}RJ$
is a cotriangular structure on $A^J$.  The cocycle $J$ is called {\bf minimal} if
$R^J$ is nondegenerate. When $A=\Os(G)$, we use the trivial initial
cotriangular form $\varepsilon\otimes\varepsilon$. By
\cite[Theorem 3.1]{G1}, every Hopf $2$-cocycle on $\Os(G)$ is gauge
equivalent to the inflation of a minimal Hopf $2$-cocycle on a closed subgroup
$H\subseteq G$. This subgroup, determined up to conjugacy, is the
\textbf{support} of $J$.

\subsection{Action groupoids and twisted action-groupoid algebras}
\label{sec:twisted-action-groupoid}
Let a group $\Gamma$ act on a set $X$.  The {\bf action groupoid}
$\Gamma\ltimes X$ is the category whose objects are the elements of $X$ and whose
arrows are
\[
 [g,x]:x\longrightarrow gx,
 \qquad g\in\Gamma,\ x\in X.
\]
Thus, two arrows $[g,x']$ and $[h,x]$ are composable precisely when $x'=hx$, and
then
\[
 [g,hx]\circ[h,x]=[gh,x].
\]
The identity arrow at $x$ is $[1,x]$, and the inverse of $[g,x]$ is
$[g^{-1},gx]$.

A normalized $2$-cocycle on the groupoid $\Gamma\ltimes X$ is a function
\[
 \omega:\Gamma\times\Gamma\times X\longrightarrow\Ctimes
\]
such that
\begin{equation}\label{eq:groupoid-cocycle}
 \omega(g,h;rx)\omega(gh,r;x)
 =\omega(h,r;x)\omega(g,hr;x)
\end{equation}
for all $g,h,r\in\Gamma$ and $x\in X$, and
$\omega(1,g;x)=\omega(g,1;x)=1$.  The corresponding {\bf twisted action-groupoid
algebra} $\mathbb C^{\omega}[\Gamma\ltimes X]$ is the vector space of finitely
supported linear combinations of the arrows, with multiplication
\begin{equation}\label{eq:twisted-groupoid-product}
 [g,x'][h,x]
 =
 \begin{cases}
  \omega(g,h;x)[gh,x],&x'=hx,\\
  0,&x'\neq hx.
 \end{cases}
\end{equation}
Equation~\eqref{eq:groupoid-cocycle} is exactly the associativity condition for
\eqref{eq:twisted-groupoid-product}.  The elements
\[
 e_x:=[1,x],\qquad x\in X,
\]
are pairwise orthogonal idempotents; if $X$ is finite, their sum is the identity
of $\mathbb C^{\omega}[\Gamma\ltimes X]$.

\begin{lemma}
\label{lem:transitive-groupoid-reduction}
Assume that the action of $\Gamma$ on the finite set $X$ is transitive.  Fix
$x_0\in X$, let
$H=\Gamma_{x_0}:=\{g\in\Gamma\mid gx_0=x_0\}$,
$\omega_0(g,h):=\omega(g,h;x_0)$ $(g,h\in H)$,
and set $\mathcal A:=\mathbb C^{\omega}[\Gamma\ltimes X]$.  Then the following hold:
\begin{enumerate}
\item
$e_{x_0}\mathcal A e_{x_0}\cong\mathbb C^{\omega_0}[H]$.
\item
$e_{x_0}$ is a full idempotent and, after choosing one arrow from
$x_0$ to every object of $X$, there is an algebra isomorphism
\begin{equation}\label{eq:transitive-groupoid-matrix-reduction}
 \mathcal A
 \cong
 \operatorname{Mat}_{|X|}\!\left(\mathbb C^{\omega_0}[H]\right).
\end{equation}
\item
Restriction to the corner gives an equivalence
${\rm Mod}(\mathcal A)
 \simeq
 {\rm Mod}(\mathbb C^{\omega_0}[H])$.
\end{enumerate}
\end{lemma}

\begin{proof}
The corner $e_{x_0}\mathcal A e_{x_0}$ is spanned by the loops $[h,x_0]$ with
$h\in H$, and their multiplication is
\[
 [g,x_0][h,x_0]=\omega_0(g,h)[gh,x_0].
\]
This proves the identification with $\mathbb C^{\omega_0}[H]$.

For each $x\in X$, choose $r_x\in\Gamma$ with $r_xx_0=x$, $r_{x_0}=1$, and let
$p_x:=[r_x,x_0]$.
The two products
$[r_x^{-1},x]p_x$ and $p_x[r_x^{-1},x]$
are nonzero scalar multiples of $e_{x_0}$ and $e_x$, respectively.  The cocycle
identity shows that the two scalars are equal.  Hence one can rescale
$[r_x^{-1},x]$ to an arrow $q_x$ such that
\begin{equation}\label{eq:connecting-arrows}
 q_xp_x=e_{x_0},
 \qquad
 p_xq_x=e_x.
\end{equation}
In particular, $e_x=p_xe_{x_0}q_x$ belongs to
$\mathcal A e_{x_0}\mathcal A$ for every $x$, so $e_{x_0}$ is full.

Let $E_{x',x}$ denote the matrix unit indexed by $x',x\in X$.  For
$a\in e_{x'}\mathcal A e_x$, define
\begin{equation}\label{eq:explicit-groupoid-matrix-map}
 \Phi(a):=E_{x',x}\otimes q_{x'}ap_x.
\end{equation}
The middle factor belongs to $e_{x_0}\mathcal A e_{x_0}$.  Using
\eqref{eq:connecting-arrows}, one checks that
$\Phi(ab)=\Phi(a)\Phi(b)$ 
whenever $a$ and $b$ are composable, while both sides vanish otherwise.  The
inverse map is
$E_{x',x}\otimes b\longmapsto p_{x'}bq_x$,
$b\in e_{x_0}\mathcal A e_{x_0}$.
This proves \eqref{eq:transitive-groupoid-matrix-reduction}.

Finally, (3) follows, for example, from the mutually inverse functors
$N\longmapsto e_{x_0}N$ and 
$W\longmapsto \mathcal A e_{x_0}
       \otimes_{e_{x_0}\mathcal A e_{x_0}}W$.
\end{proof}

\subsection{Actions and equivariantizations}\label{sec:actions}

Let $K$ be a {\bf finite} group and $\C$ a {\bf semisimple} tensor category (not necessarily finite). Recall (see, e.g., \cite[Definition 2.7.1]{EGNO}) that an action of $K$
on $\C$ consists of tensor autoequivalences $A_k$ and tensor natural
isomorphisms
$\gamma_{k,\ell}:A_kA_\ell\xrightarrow{\cong}A_{k\ell}$, $k,\ell\in K$,
satisfying the unit and pentagon axioms.

Recall (see, e.g., \cite[Definition 2.7.2]{EGNO}) that a $K$-equivariant object is a pair $(X,u)$, where $X\in\C$ and
$u_k:A_k(X)\xrightarrow{\cong}X$
satisfy
\begin{equation}\label{eq:equivariant-object}
 u_kA_k(u_\ell)=u_{k\ell}\gamma_{k,\ell}(X).
\end{equation}
The category of equivariant objects is the equivariantization $\C^K$.
For the trivial action on $\Vect$, one recovers $\Vect^K\simeq\Rep(K)$.

\subsection{Equivariant module categories}\label{sec:equivariant-modules}

Let $\M$ be a left $\C$-module category.  For $k\in K$, define the twisted module
category $\M^k$ to have the same underlying category as $\M$, and action
\begin{equation}\label{eq:module-twist}
 X\otimes^k M:=A_{k^{-1}}(X)\otimes M.
\end{equation}
With this convention, $k\mapsto\M^k$ is a left action on equivalence classes of
$\C$-module categories.  More precisely, the coherence $\gamma$ induces canonical
module equivalences
\[
 \Gamma_{k,\ell}:(\M^\ell)^k\xrightarrow{\simeq}\M^{k\ell};
\]
the module structure of $\Gamma_{k,\ell}$ is induced by
$\gamma_{\ell^{-1},k^{-1}}$.

Let $L\subseteq K$ be a subgroup.  Recall \cite{ENO} that an $L$-equivariant structure on $\M$ consists of module
equivalences
\[
 U_k:\M^k\xrightarrow{\simeq}\M,
 \qquad k\in L,
\]
and module natural isomorphisms
\begin{equation}\label{eq:mu-abstract}
 \mu_{k,\ell}:U_k\circ(U_\ell)^k
 \xrightarrow{\cong}
 U_{k\ell}\circ\Gamma_{k,\ell},
\end{equation}
subject to the usual unit and pentagon conditions.  We always use normalized data:
$U_1=\id$, and the unit components of $\mu$ are identities.

An equivariant object is a pair $(M,v)$ with
$v_k:U_k(M)\xrightarrow{\cong}M$
satisfying
\begin{equation}\label{eq:equivariant-module-object}
 v_kU_k(v_\ell)=v_{k\ell}\mu_{k,\ell}(M).
\end{equation}
The resulting finite semisimple category is denoted $\M^L$.  Recall \cite{ENO} that it carries a natural
$\C^K$-module structure; on underlying objects the action is the original action
$X\otimes M$ in $\M$.

For the trivial action on $\Vect$, an equivariant structure is determined by a
normalized cocycle $\mu\in Z^2(L,\Ctimes)$, and
$\Vect^L\simeq\Rep(L,\mu)$.

\subsection{Simple objects and rank}\label{sec:orbit-stabilizer}

Let $L$ act equivariantly on a finite semisimple category $\M$.  Then the functors $U_k$
permute $\Irr(\M)$.  For $M\in\Irr(\M)$, let
\[
 L_M:=\{k\in L\mid U_k(M)\cong M\};
\]
it is a subgroup of $L$, hence of $K$. After choosing isomorphisms $U_k(M)\cong M$, $k\in L_M$, the coherence data
induces a well-defined class
$[\alpha_M]\in H^2(L_M,\Ctimes)$.  
We call a $2$-cocycle on a finite group
{\bf nondegenerate} if its twisted group algebra is simple.

We recall the standard orbit--stabilizer description. The proof fixes the
projective-cocycle convention used below.

\begin{theorem}\label{thm:orbit-stabilizer}
For each $L$-orbit $\mathcal O\subseteq\Irr(\M)$, choose $M\in\mathcal O$.  Then
\[
 \M^L\simeq
 \bigoplus_{\mathcal O\in L\backslash\Irr(\M)}
 \Rep(L_M,\alpha_M),
\]
as an abelian category. Thus,
\[
 \rank(\M^L)=
 \sum_{\mathcal O\in L\backslash\Irr(\M)}
 \#\Irr\bigl(\mathbb{C}^{\alpha_M}[L_M]\bigr).
\]
In particular, $\M^L$ has rank one if and only if $L$ acts transitively on
$\Irr(\M)$ and the twisted group algebra $\mathbb{C}^{\alpha_M}[L_M]$ is simple for one,
equivalently every, $M\in\Irr(\M)$.
\end{theorem}

\begin{proof}
Fix an orbit $\mathcal O$, and choose a representative simple object $M_x$ for
each $x\in\mathcal O$.  Choose isomorphisms
\[
 c_{k,x}:U_k(M_x)\xrightarrow{\cong}M_{kx}.
\]
Since every $M_x$ is simple, comparison of the two isomorphisms from
$U_kU_\ell(M_x)$ to $M_{k\ell x}$ gives a unique scalar
$\omega(k,\ell;x)\in\Ctimes$, which we normalize by
\begin{equation}\label{eq:groupoid-cocycle-from-equivariance}
 c_{k,\ell x}\,U_k(c_{\ell,x})
 =\omega(k,\ell;x)^{-1}
   c_{k\ell,x}\,\mu_{k,\ell}(M_x).
\end{equation}
The pentagon for $\mu$ says precisely that $\omega$ satisfies
\eqref{eq:groupoid-cocycle}; thus $\omega$ is a normalized $2$-cocycle on the
action groupoid $L\ltimes\mathcal O$ of
\S\ref{sec:twisted-action-groupoid}.

Every object of $\M$ supported on $\mathcal O$ has the form
\[
 N=\bigoplus_{x\in\mathcal O}V_x\otimes M_x.
\]
Under the chosen maps $c_{k,x}$, an equivariant structure on $N$ is equivalent
to a family of linear maps
$\rho(k,x):V_x\longrightarrow V_{kx}$
satisfying
\begin{equation}\label{eq:groupoid-representation-relation}
 \rho(k,\ell x)\rho(\ell,x)
 =\omega(k,\ell;x)\rho(k\ell,x).
\end{equation}
Relation~\eqref{eq:groupoid-representation-relation} says exactly that the arrow
$[k,x]$ acts from $V_x$ to $V_{kx}$, and that these arrow actions define a module
over $\mathbb C^{\omega}[L\ltimes\mathcal O]$.  Hence the full subcategory of
$\M^L$ supported on $\mathcal O$ is equivalent to
${\rm Mod}(\mathbb C^{\omega}[L\ltimes\mathcal O])$.

Choose $M=M_{x_0}$ in the orbit $\mathcal O$.  The isotropy group at $x_0$ is $L_M$, and the
restriction of $\omega$ to its loops represents the class $[\alpha_M]$.  By
Lemma~\ref{lem:transitive-groupoid-reduction},
\[
 {\rm Mod}(\mathbb C^{\omega}[L\ltimes\mathcal O])
 \simeq
 {\rm Mod}(\mathbb C^{\alpha_M}[L_M])
 =\Rep(L_M,\alpha_M).
\]
Taking the direct sum over the orbits proves the categorical decomposition.  The
rank formula follows by counting the simple modules of the twisted stabilizer
algebras, and the rank is one exactly under the two conditions stated in the
theorem.
\end{proof}

\begin{remark}\label{rem:no-rank-inequality}
There is no general inequality comparing $\rank(\M^L)$ and $\rank(\M)$.  For
example, if $\mathbb{Z}/2\mathbb{Z}$ interchanges the two simple summands of $\Vect\oplus\Vect$, then
$(\Vect\oplus\Vect)^{\mathbb{Z}/2\mathbb{Z}}\simeq\Vect$. \qed
\end{remark}

\subsection{Module categories over an equivariantization}\label{sec:module-classification-general}

Let $\C\rtimes K$ be the crossed-product tensor category (see, e.g., \cite[Definition 4.15.5]{EGNO}).  Recall that the category $\C$ is a
left $\C\rtimes K$-module category, and the standard duality construction gives a tensor equivalence
\[
 (\C\rtimes K)^*_{\C}=\Fun_{\C\rtimes K}(\C,\C)\simeq\C^K.
\]
Explicitly, an equivariant object $(X,u)$ acts on $\C$ by the module endofunctor
$X\otimes-$; the equivariant structure supplies the module constraints for the
homogeneous components of $\C\rtimes K$.  Evaluation at the tensor unit gives the
inverse construction. (For the fusion case see, e.g., \cite[Example 7.12.25]{EGNO}.)

\begin{proposition}\label{prop:equivariantization-modules}
Let $T$ be a complex torus, let $\C:=\Rep(T)$, and let a finite group $K$ act
on $\C$. Then every finite indecomposable semisimple $\C^K$-module category is
equivalent to $\M^L$ for some subgroup $L\subseteq K$ and some indecomposable
$L$-equivariant $\C$-module category $\M$.

Moreover, two such pairs determine equivalent $\C^K$-module categories precisely when they
are related by conjugation in $K$ and an equivariant $\C$-module equivalence.
\end{proposition}

\begin{proof}
We use the standard finite-group descent construction underlying
\cite[Proposition 5.4]{ENO}; see also \cite[Section 8.23]{EGNO}. We spell
out the finiteness point because $\C$ need not be a fusion category.
Let $\mathscr D:=\C^K$, and let $A=\operatorname{Fun}(K)$ be the regular
commutative separable algebra in the canonical central copy of $\Rep(K)$
in $\mathscr D$. The category $\mathscr D_A$ of $A$-modules is tensor
equivalent to $\C$.

Let $\mathscr P$ be a finite semisimple $\mathscr D$-module category.
The category $A\text{-}\operatorname{Mod}(\mathscr P)$ is finite
semisimple and is a module category over $\mathscr D_A\simeq\C$.
Indeed, separability makes every $A$-module a direct summand of a free
module. The free modules on the finitely many simple objects of
$\mathscr P$ have finite length and finite-dimensional endomorphism
algebras, so their summands supply only finitely many simple $A$-modules.
Right translation on $A$ gives a compatible $K$-action on this category.
Finite-group descent identifies
\[
 \bigl(A\text{-}\operatorname{Mod}(\mathscr P)\bigr)^K
 \simeq\mathscr P.
\]
The forward construction sends $P$ to the free module $A\otimes P$ with
its translation equivariance. The inverse takes the invariant summand of
an equivariant $A$-module; the averaging idempotent exists in
characteristic zero. These constructions use only finite sums and
splittings of idempotents, and respect the module structures.

Decompose $A\text{-}\operatorname{Mod}(\mathscr P)$ into indecomposable
$\C$-module summands. The group $K$ permutes them. If $\mathscr P$ is
indecomposable, this action is transitive, since a union of orbits would
descend to a nontrivial direct-sum decomposition of $\mathscr P$.
Choose a summand $\M$ and let $L\subseteq K$ be its stabilizer.
Restriction of the $K$-action gives an $L$-equivariant structure on $\M$.
An equivariant object on the sum of the translates of $\M$ is determined
by its component in $\M$ and its $L$-equivariance. Thus
$\mathscr P\simeq\M^L$. 

Conversely, induction from an
$L$-equivariant summand followed by descent reconstructs this category.
Choosing a different summand conjugates $L$, and the constructions respect
equivalences. This proves both assertions.
\end{proof}

\section{The reductive case}\label{sec:reductive}

From now on, let $G$ be an affine algebraic reductive group whose identity
component is a torus $T$, and let $K:=G/T$.

\subsection{\texorpdfstring{The tensor equivalence $\Rep(G)\simeq\Rep(T)^K$}{The tensor equivalence for Rep(G)}}
\label{sec:tensor-equivalence}

Consider the exact sequence
\begin{equation}\label{eq:extension}
 1\rightarrow T\rightarrow G\xrightarrow{\pi}K\rightarrow 1.
\end{equation}
Choose a set-theoretic section $\sigma:K\xrightarrow{1:1} G$ with $\sigma(1)=1$.  Define
\[
 \theta_k(t):=\sigma(k)t\sigma(k)^{-1},
 \qquad
 \eta(k,\ell):=\sigma(k)\sigma(\ell)\sigma(k\ell)^{-1}
\]
for all $k,\ell\in K$ and $t\in T$.

\begin{lemma}\label{lem:extension-data}
The maps $\theta$ and $\eta$ have the following properties:
\begin{enumerate}
\item $\theta:K\to\Aut(T)$ is an action, independent of the chosen section.
\item $\eta\in Z^2(K,T)$ for this action; explicitly,
\[
 \theta_k(\eta(\ell,m))\eta(k,\ell m)
 =\eta(k,\ell)\eta(k\ell,m)
\]
for all $k,\ell,m\in K$. Its cohomology class is independent of $\sigma$.
\item 
If $N:=\{t\in T\mid t^{|K|}=1\}$, then $[\eta]$ lies in the image of
\[
 H^2(K,N)\longrightarrow H^2(K,T).
\]
After changing $\sigma$ by a $T$-valued $1$-cochain, one may therefore choose an
$N$-valued representative of the extension class.
\item 
The map $T\times K\to G$, $(t,k)\mapsto t\sigma(k)$, is an isomorphism of
varieties, and the multiplication is
$(t,k)(s,\ell)=\bigl(t\theta_k(s)\eta(k,\ell),k\ell\bigr)$.
\end{enumerate}
\end{lemma}

\begin{proof}
(1), (2), and (4) follow from associativity in $G$
and the commutativity of $T$.  

(3) Use the exact sequence of
$K$-modules
\[
 1\longrightarrow N\longrightarrow T
 \xrightarrow{t\mapsto t^{|K|}}T\longrightarrow1.
\]
The power map is surjective because $T$ is a complex torus, while multiplication by
$|K|$ annihilates $H^2(K,T)$.  Thus the long exact cohomology sequence gives the claim.
\end{proof}

Let $\widehat{T}$ be the {\bf group of characters} of $T$. Recall that $\Rep(T)$ is a semisimple pointed tensor category with simple objects 
$\{\delta_x\mid x\in \widehat{T}\}$. 
Define
\[
 x^k:=x\circ\theta_{k^{-1}},\qquad x\in \widehat{T},\quad k\in K.
\]
For $X\in\Rep(T)$, let $A_k(X)$ be the pullback of $X$ along $\theta_{k^{-1}}$.
Thus $A_k(\delta_x)=\delta_{x^k}$.  Define a tensor natural isomorphism
\begin{equation}\label{eq:gamma-extension}
 \gamma_{k,\ell}(X)
 :=\rho_X\!\left(\theta_{(k\ell)^{-1}}(\eta(k,\ell))\right)
 :A_kA_\ell(X)\xrightarrow{\cong}A_{k\ell}(X).
\end{equation}
The cocycle identity for $\eta$ is exactly the pentagon identity for $\gamma$, so
$(A,\gamma)$ is an action of $K$ on $\Rep(T)$.

The following is the standard equivariantization associated with the
extension, written with the preceding choice of coherence maps.

\begin{theorem}\label{thm:tensor-equivalence}
The functor
\[
 \Phi:\Rep(T)^K\longrightarrow\Rep(G),
 \qquad
 (X,u)\longmapsto\bigl[t\sigma(k)\mapsto\rho_X(t)u_k\bigr],
\]
is a tensor equivalence.
\end{theorem}

\begin{proof}
Since $u_k:A_k(X)\xrightarrow{\cong} X$ is $T$-linear,
$u_k\rho_X(s)=\rho_X(\theta_k(s))u_k$.
Moreover, \eqref{eq:equivariant-object} and \eqref{eq:gamma-extension} give
$u_ku_\ell=\rho_X(\eta(k,\ell))u_{k\ell}$.
Together with Lemma~\ref{lem:extension-data}(4), these identities show that the
displayed formula defines a representation of $G$.

Conversely, if $(X,\pi)$ is a $G$-representation, restrict it to $T$ and set
$u_k:=\pi(\sigma(k))$.  Then $u_k:A_k(X)\to X$ is $T$-linear, and
$u_ku_\ell=\pi(\eta(k,\ell))u_{k\ell}
 =u_{k\ell}\gamma_{k,\ell}(X)$.
Thus $(X|_T,u)$ is a $K$-equivariant object.  

Finally, it is straightforward to verify that the two constructions above are mutually
inverse and preserve tensor products.
\end{proof}

For completeness, we record the corresponding description of irreducible
representations of $G$.  For $x\in \widehat{T}$, let
\[
 K_x:=\{k\in K\mid x^k=x\},
 \qquad
 \alpha_x(k,\ell)=x(\eta(k,\ell)),\quad k,\ell\in K_x.
\]
The restriction of $\eta$ to $K_x$ evaluates to the scalar cocycle $\alpha_x$.

\begin{proposition}\label{prop:simple-G-modules}
Isomorphism classes of simple modules over $G$ are parametrized by $K$-orbits in $\widehat{T}$
together with simple $\alpha_x$-projective representations of $K_x$.  If $W$ is
such a projective representation, the corresponding simple $G$-module is
\[
 \Ind_{G_x}^{G}(W\otimes\delta_x),
 \qquad G_x:=\pi^{-1}(K_x),
\]
where $t\sigma(k)\in G_x$ acts on $W\otimes\delta_x$ by
$x(t)\rho_W(k)$.
\end{proposition}

\begin{proof}
The formula defines a representation of $G_x$ because $x$ is $K_x$-invariant and
$\rho_W(k)\rho_W(\ell)=\alpha_x(k,\ell)\rho_W(k\ell)$.  The claim now follows
from Theorem~\ref{thm:tensor-equivalence} and the orbit--stabilizer description of
simple objects in an equivariantization.
\end{proof}

\subsection{\texorpdfstring{Finite module categories over $\Rep(T)$}{Finite module categories over Rep(T)}}
\label{sec:module-cats-torus}

Let $S\subseteq \widehat{T}$ be a subgroup of {\bf finite} index.  A normalized
{\bf module $2$-cocycle} on the transitive $\widehat{T}$-set $\widehat{T}/S$ is a function
\[
 \Psi:\widehat{T}\times \widehat{T}\times (\widehat{T}/S)\longrightarrow\Ctimes,
 \qquad (x,z,y)\longmapsto\Psi(x,z;y),
\]
satisfying
\begin{equation}\label{eq:module-cocycle}
 \Psi(x,z;wy)\Psi(xz,w;y)
 =\Psi(z,w;y)\Psi(x,zw;y)
\end{equation}
for all $x,z,w\in \widehat{T}$ and $y\in \widehat{T}/S$, together with
$\Psi(1,x;y)=\Psi(x,1;y)=1$.

Define $\M(S,\Psi)$ to be the finite semisimple category with simple objects
$$\{\delta_y\mid y\in \widehat{T}/S\},$$
action
\[
 \delta_x\otimes\delta_y=\delta_{xy},
\]
and module associativity constraint
\[
 m_{x,z,y}=\Psi(x,z;y)\id_{\delta_{xzy}}.
\]
Equation \eqref{eq:module-cocycle} is precisely the module pentagon.

\begin{proposition}[\cite{G2}; see also \cite{EGNO}]\label{prop:module-cats-torus}
Every finite indecomposable semisimple $\Rep(T)$-module category is equivalent to
$\M(S,\Psi)$ for some finite-index subgroup $S\subseteq \widehat{T}$ and some normalized module
$2$-cocycle $\Psi$.  Its equivalence class depends only on the cohomology class of
$\Psi$.  By Shapiro's lemma, these cohomology classes are naturally identified with
$H^2(S,\Ctimes)$.
\end{proposition}

\begin{proof}
Each invertible simple object $\delta_x$ acts by an autoequivalence, and hence $\widehat{T}$
acts on the finite set of simple objects of the module category.  Indecomposability
is equivalent to transitivity of this action.  Choose a simple object $M_0$, let
$S\subseteq \widehat{T}$ be its stabilizer, and identify the simple objects with $\widehat{T}/S$.
After choosing normalized isomorphisms
$c_{x,y}:\delta_x\otimes M_y\xrightarrow{\cong}M_{xy}$,
the module associativity constraint becomes multiplication by scalars
$\Psi(x,z;y)$.  The module pentagon and unit axioms give
\eqref{eq:module-cocycle} and the stated normalizations.  

Conversely, those
identities define a module category.  Replacing the maps $c_{x,y}$ by scalar
multiples changes $\Psi$ by a module coboundary, and every such coboundary arises
this way.  

Finally, the coefficient module is the coinduced $\widehat{T}$-module
$\Map(\widehat{T}/S,\Ctimes)$,
so Shapiro's lemma identifies its second cohomology with
$H^2(S,\Ctimes)$.
\end{proof}

\subsubsection{Explicit equivariant structures}\label{sec:explicit-equivariant-data}

Fix $\M(S,\Psi)$ and a subgroup $L\subseteq K$.  We use the action on $\widehat{T}$ introduced in
\S\ref{sec:tensor-equivalence}.  For $k\in L$, let
\[
 \Psi^k(x,z;y):=\Psi(x^{k^{-1}},z^{k^{-1}};y),
 \qquad
 \gamma_{\ell^{-1},k^{-1}}(\delta_x)
 :=\bar\gamma_{k,\ell}(x)\,\id_{\delta_{x^{(k\ell)^{-1}}}}.
\]
Thus $\M(S,\Psi)^k$ has action
$\delta_x\otimes^k\delta_y=\delta_{x^{k^{-1}}y}$ and associator $\Psi^k$.

\begin{theorem}\label{thm:explicit-equivariant-data}
The category $\M(S,\Psi)$ admits an $L$-equivariant structure only if $S$ is
$L$-stable.  If $S$ is $L$-stable, normalized $L$-equivariant structures are
parametrized by data
\[
 U_k\in\operatorname{Sym}(\widehat{T}/S),
 \qquad
 \tau_k:\widehat{T}\times (\widehat{T}/S)\to\Ctimes,
 \qquad
 \mu_{k,\ell}:\widehat{T}/S\to\Ctimes,
\]
for $k,\ell\in L$, satisfying the following equations for all
$i,k,\ell\in L$, $x,z\in \widehat{T}$, and $y\in \widehat{T}/S$:
\begin{equation}
 U_1=\id,\quad
 U_kU_\ell=U_{k\ell},\quad
 U_k(xy)=x^kU_k(y),
 \label{eq:U-data}
 \end{equation}
 \begin{equation}
 \tau_1(x,y)=1,\quad
 \tau_k(1,y)=1,
 \label{eq:tau-normalized}
\end{equation}
\begin{equation}\label{eq:tau-coherence}
 \tau_k(z,y)\tau_k(x,z^{k^{-1}}y)\Psi^k(x,z;y)
 =\Psi(x,z;U_k(y))\tau_k(xz,y),
\end{equation}
\begin{equation}\label{eq:mu-module}
 \bar\gamma_{k,\ell}(x)\tau_{k\ell}(x,y)
 \mu_{k,\ell}(x^{(k\ell)^{-1}}y)
 =\mu_{k,\ell}(y)\tau_k(x,U_\ell(y))
 \tau_\ell(x^{k^{-1}},y),
 \end{equation}
\begin{align}
 \mu_{1,k}(y)&=\mu_{k,1}(y)=1,
 \label{eq:mu-normalized}\\
 \mu_{i,k}(U_\ell(y))\mu_{ik,\ell}(y)
 &=\mu_{k,\ell}(y)\mu_{i,k\ell}(y).
 \label{eq:mu-pentagon}
\end{align}
Here $U_k$ is the permutation induced by the underlying functor of
$$U_k:\M(S,\Psi)^k\to\M(S,\Psi),$$
while $\tau_k$ and $\mu_{k,\ell}$ are the scalar
components of its module structure and coherence isomorphisms.
\end{theorem}

\begin{proof}
A module equivalence $U_k:\M(S,\Psi)^k\xrightarrow{\simeq}\M(S,\Psi)$ permutes the simple objects.
Its module constraint has components
$U_k(\delta_{x^{k^{-1}}y})\longrightarrow
 \delta_x\otimes U_k(\delta_y)$.
Equality of the two simple labels is equivalent to
$U_k(xy)=x^kU_k(y)$, which also implies that $S$ is $L$-stable.  The module-functor
pentagon and unit axiom are exactly \eqref{eq:tau-normalized} and
\eqref{eq:tau-coherence}.

The natural transformation
$\mu_{k,\ell}:U_k(U_\ell)^k
 \longrightarrow U_{k\ell}\Gamma_{k,\ell}$
can have a nonzero component only when $U_kU_\ell=U_{k\ell}$.  The condition that
$\mu_{k,\ell}$ be a module natural transformation gives
\eqref{eq:mu-module}; the appearance of $\bar\gamma_{k,\ell}$ comes from the module
structure of $\Gamma_{k,\ell}$.  Finally, its pentagon is
\eqref{eq:mu-pentagon}.  

Conversely, the displayed equations reconstruct all the
required module functors and coherence maps.
\end{proof}

\subsubsection{Simple objects and the rank-one criterion}
\label{sec:rank-one-criterion}

For $y\in \widehat{T}/S$, let
\[
 L_y:=\{k\in L\mid U_k(y)=y\},
 \qquad
 \mu_y(k,\ell):=\mu_{k,\ell}(y),\quad k,\ell\in L_y.
\]
Equation \eqref{eq:mu-pentagon} shows that $\mu_y\in Z^2(L_y,\Ctimes)$.

\begin{corollary}\label{thm:simple-equivariantized-module}
As an abelian category,
\[
 \M(S,\Psi)^{L}
 \simeq
 \bigoplus_{\mathcal O\in L\backslash (\widehat{T}/S)}
 \Rep(L_y,\mu_y),
 \qquad y\in\mathcal O.
\]
In particular, $\M(S,\Psi)^L$ has rank one if and only if
\begin{enumerate}
\item $L$ acts transitively on $\widehat{T}/S$ through $U$; and
\item $\mu_y$ is nondegenerate on $L_y$ for one, equivalently every, $y\in \widehat{T}/S$.
\end{enumerate}
\end{corollary}

\begin{proof}
Apply Theorem~\ref{thm:orbit-stabilizer} to the action of $L$ on the simple
objects $\delta_y$. The groupoid cocycle in its proof is
$\omega(k,\ell;y)=\mu_{k,\ell}(y)$ by
\eqref{eq:mu-pentagon}. Its restriction to the stabilizer at $y$ is
$\mu_y$, so the decomposition and the rank-one criterion follow.
\end{proof}

\begin{example}
\label{ex:mixed-rank-one}
Let $G:=\Gm\times \mathbb{Z}/2\mathbb{Z}$.  Then $\widehat{T}=\mathbb{Z}$ and the action of $\mathbb{Z}/2\mathbb{Z}$ on $\widehat{T}$ is
trivial.  Take $S:=2\mathbb{Z}$, so $\widehat{T}/S\cong \mathbb{Z}/2\mathbb{Z}$, take $\Psi=1$, and let the nontrivial
element $g\in \mathbb{Z}/2\mathbb{Z}$ act by translation; that is,
$U_g([m])=[m+1]$.
With $\tau=1$ and $\mu=1$, all equations in Theorem~\ref{thm:explicit-equivariant-data} hold.  The action on $\widehat{T}/S$ is transitive and the
stabilizer is trivial, so $\M(2\mathbb{Z},1)^{\mathbb{Z}/2\mathbb{Z}}$ has rank one.  Notice that
$S\neq \widehat{T}$ and $U\neq\id$. \qed
\end{example}

\subsection{\texorpdfstring{Fiber functors and Hopf $2$-cocycles}{Fiber functors and Hopf 2-cocycles}}\label{sec:fiber-functors}

Call a sextuple
$\mathcal D=(S,\Psi,L,U,\tau,\mu)$
{\bf admissible} if $S\subseteq \widehat{T}$ has finite index, $L\subseteq K$ is a subgroup, and the remaining data
satisfy Theorem~\ref{thm:explicit-equivariant-data}.  It is {\bf rank-one
admissible} if the two conditions in Corollary~\ref{thm:simple-equivariantized-module} hold.

Two admissible data are called {\bf equivalent} if, after conjugating one of the
subgroups $L$ by an element of $K$, the corresponding equivariant
$\Rep(T)$-module categories are equivalent.  This relation includes the usual
changes of cocycle representatives and of the scalar coherence cochains.

\begin{theorem}\label{thm:module-and-fiber-classification}
The following statements hold.
\begin{enumerate}
\item Equivalence classes of finite indecomposable semisimple $\Rep(G)$-module
categories are in bijection with equivalence classes of admissible data
\[
 (S,\Psi,L,U,\tau,\mu).
\]
\item Equivalence classes of fiber functors on $\Rep(G)$ are in bijection with
equivalence classes of rank-one admissible data.
\item Every fiber functor on $\Rep(G)$ is classical.
\end{enumerate}
\end{theorem}

\begin{proof}
(1) By Theorem~\ref{thm:tensor-equivalence}, Proposition~\ref{prop:equivariantization-modules}, and Proposition~\ref{prop:module-cats-torus}, every finite indecomposable module category has the
form
$\bigl(\M(S,\Psi),U,\tau,\mu\bigr)^L$,
and the stated equivalence relation is precisely the one arising from those three
results.  

(2) This follows from Corollary~\ref{thm:simple-equivariantized-module}, because fiber functors are the same as
rank-one module categories.

(3) Let $\V$ be a rank-one module category corresponding to
rank-one admissible data, and let $V$ be its unique simple object.  Write the
underlying object of $V$ in $\M(S,\Psi)$ as
\[
 \overline V=\bigoplus_{y\in \widehat{T}/S}V_y\otimes\delta_y.
\]
Transitivity implies that every $V_y$ is nonzero and that all $\dim V_y$ are equal.
For an object $M=\bigoplus_yM_y\otimes\delta_y$ of $\M(S,\Psi)$, set
\[
 \ell(M):=\sum_{y\in \widehat{T}/S}\dim M_y.
\]
If $X\in\Rep(G)$ and
$X|_T=\bigoplus_{x\in \widehat{T}}X_x\otimes\delta_x$, then the object-level action gives
\[
 \ell(X\otimes\overline V)
 =\left(\sum_{x\in \widehat{T}}\dim X_x\right)\ell(\overline V)
 =(\dim X)\ell(\overline V).
\]
Let $F_{\V}:\Rep(G)\to\Vect$ be the fiber functor defined by the action on the
rank-one category.  Since
$X\otimes V\cong F_{\V}(X)\otimes V$,
the same length function gives
\[
 (\dim F_{\V}(X))\,\ell(\overline V)
 =(\dim X)\,\ell(\overline V).
\]
Hence $\dim F_{\V}(X)=\dim X$ for every $X$.

Since $\Rep(G)$ is semisimple, an exact $\mathbb C$-linear functor to $\Vect$ is
determined up to natural isomorphism by its values on simple objects.  Choosing
vector-space isomorphisms on the simples therefore yields a natural isomorphism
between the underlying linear functor of $F_{\V}$ and the usual forgetful functor.
Thus $F_{\V}$ is classical.
\end{proof}

\begin{corollary}\label{cor:hopf-cocycle-classification}
Gauge-equivalence classes of Hopf $2$-cocycles on $\mathscr{O}(G)$ are in bijection
with equivalence classes of rank-one admissible data
$(S,\Psi,L,U,\tau,\mu)$.
\end{corollary}

\begin{proof}
Transport the tensor structure of a fiber functor along the natural isomorphism
with the forgetful functor supplied by Theorem~\ref{thm:module-and-fiber-classification}(3), and use
\S\ref{sec:hopf-cocycles}.  Monoidal natural isomorphisms correspond exactly to
gauge transformations.
\end{proof}

We now study the support and minimality of a Hopf $2$-cocycle on $\mathscr O(G)$.

\begin{theorem}\label{prop:factorization}
Let $\mathcal D:=(S,\Psi,L,U,\tau,\mu)$ 
be rank-one admissible data, and let $y_0:=S\in \widehat{T}/S$.  Since $S$ fixes
$y_0$, the formula
\begin{equation}\label{eq:toral-isotropy-cocycle}
 \psi_{\mathcal D}(s,t):=\Psi(s,t;y_0),
 \qquad s,t\in S,
\end{equation}
defines a normalized cocycle $\psi_{\mathcal D}\in Z^2(S,\Ctimes)$.  Let
\[
 R_{\mathcal D}(s,t)
 :=\frac{\psi_{\mathcal D}(s,t)}{\psi_{\mathcal D}(t,s)},
 \] 
 \[
 \Lambda_{\mathcal D}
 :=\Rad(R_{\mathcal D})
 =\{s\in S\mid R_{\mathcal D}(s,t)=1\text{ for all }t\in S\},
\]
and let
\[
 T_{\mathcal D}:=\Lambda_{\mathcal D}^{\perp}
 =\{a\in T\mid s(a)=1\text{ for every }s\in\Lambda_{\mathcal D}\}.
\]
Let $J_{\mathcal D}$ be the Hopf $2$-cocycle corresponding to $\mathcal D$, and let
$H_{\mathcal D}\subseteq G$ be its support.  Then after conjugating
$H_{\mathcal D}$ in $G$, one has
\begin{equation}\label{eq:support-from-admissible-data}
 \pi(H_{\mathcal D})=L,
 \qquad
 H_{\mathcal D}\cap T=T_{\mathcal D}.
\end{equation}
Thus,
\begin{equation}\label{eq:minimality-full-data}
 J_{\mathcal D}\text{ is minimal}
 \quad\Longleftrightarrow\quad
 L=K\ \text{ and }\ \Lambda_{\mathcal D}=\{1\}.
\end{equation}
Equivalently, gauge-equivalence classes of minimal Hopf $2$-cocycles on
$\mathscr O(G)$ are parametrized by equivalence classes of admissible data
$(S,\Psi,K,U,\tau,\mu)$
for which $K$ acts transitively on $\widehat{T}/S$, the stabilizer cocycle $\mu_{y_0}$ is
nondegenerate, and the alternating bicharacter $R_{\mathcal D}$ on $S$ is
nondegenerate.
\end{theorem}

\begin{proof}
Set $G_L:=\pi^{-1}(L)$.  The module category attached to $\mathcal D$ is first a
rank-one $\Rep(G_L)\simeq\Rep(T)^L$-module category, and its $\Rep(G)$-action is
obtained by restriction along $\Rep(G)\to\Rep(G_L)$.  Thus it is enough to compute
the support of the corresponding cocycle on $G_L$.

Let $V$ be the unique simple object of $\M(S,\Psi)^L$, and let 
\[
 \overline V=\bigoplus_{y\in \widehat{T}/S}V_y\otimes\delta_y
\]
be its image in $\M(S,\Psi)$.
By transitivity, every $V_y$ is nonzero and all the spaces $V_y$ have the same
dimension, say $d$.  Let $\mathcal A_{\mathcal D}$ be the internal endomorphism
algebra of $V$ in $\operatorname{Ind}(\Rep(G_L))$.  The induction--forgetful
adjunction for equivariantization identifies its image in
$\operatorname{Ind}(\Rep(T))$ with the internal endomorphism algebra
\[
 \mathcal B_{\mathcal D}
 :=\underline{\End}_{\M(S,\Psi)}(\overline V).
\]
Explicitly, $\mathcal B_{\mathcal D}$ is the $\widehat{T}$-graded algebra with
\[
 (\mathcal B_{\mathcal D})_x
 =\bigoplus_{y\in \widehat{T}/S}\Hom(V_y,V_{xy}),
 \qquad x\in \widehat{T},
\]
and, for $f\in(\mathcal B_{\mathcal D})_x$ and
$g\in(\mathcal B_{\mathcal D})_z$, multiplication is given by
\begin{equation}\label{eq:internal-end-product}
 (f\cdot g)_y
 =\Psi(x,z;y)\,f_{zy}\circ g_y.
\end{equation}

We now explain the groupoid reduction used to compute this algebra.  Let
\[
 \mathcal G:=\widehat T\ltimes (\widehat{T}/S).
\]
Its objects are the cosets $y\in \widehat{T}/S$, and for every $x\in \widehat{T}$ and $y\in \widehat{T}/S$ it
has an arrow
\[
 [x,y]:y\longrightarrow xy.
\]
The groupoid is transitive: if $y=aS$, then $[a,y_0]$ is an arrow from
$y_0=S$ to $y$.  Its isotropy group at $y_0$ is
\[
 \operatorname{Aut}_{\mathcal G}(y_0)
 =\{[x,y_0]\mid xy_0=y_0\}
 =\{[s,y_0]\mid s\in S\}\cong S.
\]
Moreover, the module-cocycle identity \eqref{eq:module-cocycle} is precisely the
groupoid cocycle identity \eqref{eq:groupoid-cocycle}.  We may therefore form the
twisted groupoid algebra
\[
 \mathcal G_{\Psi}:=\mathbb C^{\Psi}[\widehat{T}\ltimes (\widehat{T}/S)],
\]
whose multiplication is
\begin{equation}\label{eq:Psi-twisted-groupoid-product}
 [x,zy][z,y]=\Psi(x,z;y)[xz,y].
\end{equation}
After choosing bases $V_y\cong\mathbb C^d$, formula
\eqref{eq:internal-end-product} identifies $\mathcal B_{\mathcal D}$ with the
matrix-enriched version of this algebra:
\begin{equation}\label{eq:internal-end-as-groupoid-algebra}
 \mathcal B_{\mathcal D}
 \cong
 \operatorname{Mat}_d(\mathcal G_{\Psi}).
\end{equation}
Indeed, the copy of $\operatorname{Mat}_d(\mathbb C)$ attached to the arrow
$[x,y]$ is exactly $\Hom(V_y,V_{xy})$, and
\eqref{eq:Psi-twisted-groupoid-product} reproduces
\eqref{eq:internal-end-product}.

Let $e_y:=[1,y]$.  The corner at the base point consists of the loops labelled by
$S$:
\begin{equation}\label{eq:isotropy-corner-S}
 e_{y_0}\mathcal G_{\Psi}e_{y_0}
 =\bigoplus_{s\in S}\mathbb C[s,y_0]
 \cong\mathbb C^{\psi_{\mathcal D}}[S],
\end{equation}
where the last assertion follows from
$[s,y_0][t,y_0]=\Psi(s,t;y_0)[st,y_0]$.
Choose representatives $r_y\in \widehat{T}$ with $r_yy_0=y$ and $r_{y_0}=1$, and let
$p_y,q_y$ be the normalized connecting arrows of
\eqref{eq:connecting-arrows}.  The explicit map
\[
 a\in e_{y'}\mathcal G_{\Psi}e_y
 \longmapsto
 E_{y',y}\otimes q_{y'}ap_y
\]
from \eqref{eq:explicit-groupoid-matrix-map} gives
\[
 \mathcal G_{\Psi}
 \cong
 \operatorname{Mat}_{|\widehat{T}/S|}\!\left(
   \mathbb C^{\psi_{\mathcal D}}[S]
 \right).
\]
Combining this with \eqref{eq:internal-end-as-groupoid-algebra} yields the desired
algebra isomorphism
\begin{equation}\label{eq:internal-end-matrix-algebra}
 \mathcal B_{\mathcal D}
 \cong
 \operatorname{Mat}_{d|\widehat{T}/S|}\!\left(\mathbb C^{\psi_{\mathcal D}}[S]\right).
\end{equation}

It remains to keep track of the $\widehat{T}$-grading and the center.  If $u_s$ denotes the
basis element of $\mathbb C^{\psi_{\mathcal D}}[S]$ corresponding to $s\in S$,
then the above matrix realization has the shifted grading
\begin{equation}\label{eq:shifted-matrix-grading}
 \deg(E_{y',y}\otimes u_s)=r_{y'}sr_y^{-1}.
\end{equation}
In particular, every diagonal element $E_{y,y}\otimes u_s$ has degree $s$,
because $\widehat{T}$ is abelian.  Furthermore,
$u_su_t=\psi_{\mathcal D}(s,t)u_{st}$ and 
$u_tu_s=\psi_{\mathcal D}(t,s)u_{st}$,
so
\begin{equation}\label{eq:center-of-twisted-S-algebra}
 Z\!\left(\mathbb C^{\psi_{\mathcal D}}[S]\right)
 =\bigoplus_{s\in\Lambda_{\mathcal D}}\mathbb C u_s.
\end{equation}
Since the center of a full matrix algebra consists of scalar matrices over the
center of its coefficient algebra,
\[
 Z(\mathcal B_{\mathcal D})
 =I_{d|\widehat{T}/S|}\otimes
  Z\!\left(\mathbb C^{\psi_{\mathcal D}}[S]\right).
\]
Under the inverse of the groupoid matrix map, the element
$I_{|\widehat{T}/S|}\otimes u_s$ is
\[
 \sum_{y\in \widehat{T}/S}p_yu_sq_y,
\]
a sum of loops, one at each object $y$.  Each summand has degree
$r_ysr_y^{-1}=s$.  Thus this central element is homogeneous of degree $s$, and
\eqref{eq:center-of-twisted-S-algebra} shows that the $T$-weights occurring in the
center are precisely the elements of $\Lambda_{\mathcal D}$.  In particular,
\begin{equation}\label{eq:center-weights}
 Z(\mathcal B_{\mathcal D})_x\neq0
 \quad\Longleftrightarrow\quad
 x\in\Lambda_{\mathcal D},
 \qquad
 \dim Z(\mathcal B_{\mathcal D})_1=1.
\end{equation}

By Theorem~\ref{thm:module-and-fiber-classification}, the fiber functor defined by
$\mathcal D$ is classical.  Therefore $\mathcal A_{\mathcal D}$ is the one-sided
twisted torsor algebra associated to the corresponding Hopf $2$-cocycle on
$G_L$ (up to passing to the opposite algebra, which does not change the center).
If $H_{\mathcal D}$ is its support, then
\cite[Theorem~3.5]{G1} gives an isomorphism of $G_L$-algebras
\begin{equation}\label{eq:center-support-homogeneous-space}
 Z(\mathcal A_{\mathcal D})\cong\mathscr O(G_L/H_{\mathcal D}).
\end{equation}
The degree-one part in \eqref{eq:center-weights} is one-dimensional.  On the other
hand, the $T$-orbits in $G_L/H_{\mathcal D}$ are indexed by
$L/\pi(H_{\mathcal D})$; hence
\[
 \dim\mathscr O(G_L/H_{\mathcal D})^T
 =[L:\pi(H_{\mathcal D})].
\]
Thus $\pi(H_{\mathcal D})=L$.

It follows that $T$ acts transitively on $G_L/H_{\mathcal D}$ and that
\[
 G_L/H_{\mathcal D}
 \cong T/(H_{\mathcal D}\cap T)
\]
as $T$-varieties.  Thus, the group of $T$-weights occurring in
$\mathscr O(G_L/H_{\mathcal D})$ is
$(H_{\mathcal D}\cap T)^\perp$.  Comparing this with
\eqref{eq:center-weights} yields
$(H_{\mathcal D}\cap T)^\perp=\Lambda_{\mathcal D}$,
and therefore $H_{\mathcal D}\cap T=T_{\mathcal D}$.  This proves
\eqref{eq:support-from-admissible-data}.

Finally, the lift of the cocycle from $G_L$ to $G$ has the same support
$H_{\mathcal D}$.  Hence it is minimal on $G$ if and only if
$H_{\mathcal D}=G$, which, by \eqref{eq:support-from-admissible-data}, is equivalent
to $L=K$ and $\Lambda_{\mathcal D}=\{1\}$.  The last formulation simply combines
this condition with the rank-one criterion in
Corollary~\ref{thm:simple-equivariantized-module}.
\end{proof}

\begin{remark}\label{rem:two-nondegeneracy-conditions}
The two nondegeneracy conditions in Theorem~\ref{prop:factorization} play
different roles.  Nondegeneracy of $\mu_{y_0}$ on the finite stabilizer
$L_{y_0}$ is the rank-one condition, whereas nondegeneracy of
$R_{\mathcal D}$ on the lattice $S$ is the additional condition that the torus
part of the support be all of $T$.  In particular,
$\psi_{\mathcal D}$ is extracted from the full equivariant module-category
data, and includes finite and mixed information. In the direct-product case,
where the K\"unneth decomposition defines a toral component, it can differ
from the restriction of that component to $S$. For a general extension,
a cocycle on $\Os(G)$ need not descend along $\Os(G)\twoheadrightarrow\Os(T)$,
so there is no intrinsic restriction to $\Os(T)$. \qed
\end{remark}

\subsection{Bounds and a nonsplit example}\label{sec:bounds-examples}

The support calculation and the rank-one criterion give the following
consequences without solving the full system of coherence equations.

\begin{corollary}
\label{cor:bounded-index}
Let $\mathcal D:=(S,\Psi,L,U,\tau,\mu)$ be rank-one admissible data, let $d$ be the dimension of the unique simple
module over $\mathbb C^{\mu_S}[L_S]$, and set
$n:=[\widehat T:S]$. Then
\begin{equation}\label{eq:bounded-index}
 |L|=nd^2.
\end{equation}
In particular, $n$ divides $|L|$ and $n\leq |K|$.

Writing $\widehat T/S$ additively, every permitted action has the form
\begin{equation}\label{eq:affine-orbit-action}
 U_k(y)=k\cdot y+a(k),\qquad a\in Z^1(L,\widehat T/S),
\end{equation}
where $k\cdot y$ is the action induced by conjugation on $T$. The action is
transitive exactly when $a(L)=\widehat T/S$, and
$L_S=\{k\in L\mid a(k)=0\}$.
\end{corollary}

\begin{proof}
Nondegeneracy of the stabilizer cocycle gives
$\mathbb C^{\mu_S}[L_S]\cong\operatorname{Mat}_d(\mathbb C)$ as algebras, so
$|L_S|=d^2$. Transitivity gives $n=[L:L_S]$, proving
\eqref{eq:bounded-index}.

Set $a(k):=U_k(0)$. The last identity in \eqref{eq:U-data} becomes
$U_k(x+y)=k\cdot x+U_k(y)$, which gives
\eqref{eq:affine-orbit-action}. The identity $U_kU_\ell=U_{k\ell}$ is
then exactly
\[
 a(k\ell)=a(k)+k\cdot a(\ell).
\]
The orbit of $0$ is $a(L)$, and the stabilizer of $0$ is the indicated
subgroup.
\end{proof}

Since $\widehat T$ is a finitely generated free abelian group, it has only
finitely many subgroups of index at most $|K|$. These can be enumerated by
Hermite normal form. For each $L\subseteq K$ and each $L$-stable $S$, there
are also only finitely many cocycles $a:L\to\widehat T/S$. Thus the lattice
and orbit choices in the classification reduce to finite sets. The cochains
$\Psi,\tau,\mu$ still have to satisfy the remaining compatibility equations.

\begin{corollary}
\label{cor:one-dimensional-support}
Suppose that $G^\circ\cong\Gm$, and let $K:=G/G^\circ$. Every Hopf
$2$-cocycle on $\Os(G)$ is gauge equivalent to the inflation of a minimal
Hopf $2$-cocycle on a finite subgroup $H\subseteq G$, with
$|H|\leq |K|^2$.
In particular, $G$ admits no minimal Hopf $2$-cocycle.
\end{corollary}

\begin{proof}
Choose rank-one data for the given cocycle. Here $\widehat T\cong\mathbb Z$,
so $S$ is cyclic and every alternating bicharacter on $S$ is trivial.
Thus, $\Lambda_{\mathcal D}=S$. Theorem~\ref{prop:factorization}
gives, after conjugacy,
$H\cap T=S^\perp,\qquad \pi(H)=L$.
The first group is finite of order $n=[\widehat T:S]$. Thus
\[
 |H|=n|L|=n^2d^2\leq |K|^2
\]
by Corollary~\ref{cor:bounded-index}. The inflation assertion is the support
theorem \cite[Theorem 3.1]{G1}. Since $G$ is positive dimensional and $H$
is finite, the original cocycle cannot be minimal on $G$.
\end{proof}

\begin{example}
\label{ex:split-nonsplit}
For $c\in\{1,-1\}\subseteq\Gm$, let $G_c:=T\sqcup Tj$, where $T=\Gm$ and
\[
 jtj^{-1}=t^{-1},\qquad j^2=c.
\]
Equivalently, on $T\times\{0,1\}$ the multiplication is
\[
 (t,\epsilon)(s,\delta)
 =\bigl(t s^{(-1)^\epsilon}c^{\epsilon\delta},\,
          \epsilon+\delta\pmod2\bigr).
\]
The condition $c=c^{-1}$ makes this an algebraic group. The group $G_1$ is
$\Gm\rtimes \mathbb{Z}/2\mathbb{Z}$. The group $G_{-1}$ is the normalizer of the diagonal
torus in $\operatorname{SL}_2(\mathbb C)$, with
\[
 t\longleftrightarrow
 \begin{pmatrix}t&0\\0&t^{-1}\end{pmatrix},\qquad
 j=\begin{pmatrix}0&-1\\1&0\end{pmatrix}.
\]
It is nonsplit because every $tj\notin T$ satisfies $(tj)^2=-1$.
Both groups have component group $\mathbb{Z}/2\mathbb{Z}$ acting on $\widehat T=\mathbb Z$ by
inversion, but their extension cocycles satisfy $\eta(g,g)=c$ for the
nontrivial element $g\in \mathbb{Z}/2\mathbb{Z}$.

We claim that $G_1$ has exactly two gauge classes of Hopf $2$-cocycles,
whereas $G_{-1}$ has only the trivial class. Indeed, if $L=1$,
\eqref{eq:bounded-index} forces $n=d=1$, and the proof of
Corollary~\ref{cor:one-dimensional-support} gives trivial support. If
$L=\mathbb{Z}/2\mathbb{Z}$, the same equation forces $n=2$ and $d=1$. A nontrivial support
must therefore have order four and satisfy
\[
 H\cap T=\mu_2,\qquad \pi(H)=\mathbb{Z}/2\mathbb{Z}.
\]
In $G_{-1}$ such a subgroup is cyclic of order four, since the square of
any of its elements outside $T$ is $-1$. A cyclic group has no
nondegenerate projective cocycle. The finite-group support classification
\cite{Mo,G1} excludes this case.

In $G_1$ every possible support is
$H=\langle-1,tj\rangle\cong \mathbb{Z}/2\mathbb{Z}\times \mathbb{Z}/2\mathbb{Z}$. These subgroups are all
conjugate by $T$, because
\[
 s(tj)s^{-1}=s^2tj
\]
and the square map on $T$ is surjective. There is exactly one nondegenerate
cohomology class on the character group of $\mathbb{Z}/2\mathbb{Z}\times \mathbb{Z}/2\mathbb{Z}$. With its
characters written as $(a,b)\in\mathbb F_2^2$, a representative is
\[
 J_H\bigl((a,b),(u,v)\bigr)=(-1)^{av}.
\]
Its alternating bicharacter is $(-1)^{av-ub}$ and has trivial radical.
Inflating $J_H$ along $\Os(G_1)\twoheadrightarrow\Os(H)$ gives the unique
nontrivial gauge class on $G_1$.

The distinction can also be read directly from the coherence equations.
For $L=\mathbb{Z}/2\mathbb{Z}$ and $S=2\mathbb Z$, one can take $\Psi=1$, since
$H^2(2\mathbb Z,\Ctimes)=0$. Transitivity forces
$U_g(y)=y+1$ on $\mathbb Z/2\mathbb Z$. In additive notation for
$\widehat T$, one has $\bar\gamma_{g,g}(m)=c^m$. Taking $i=k=\ell=g$
in \eqref{eq:mu-pentagon} shows that
$\mu_{g,g}(y+1)=\mu_{g,g}(y)$, so $\mu_{g,g}$ is constant. On the other
hand, \eqref{eq:tau-coherence}, with $x=1$ and $z=-1$, gives
\[
 \tau_g(-1,y)\tau_g(1,y+1)=1.
\]
Equation~\eqref{eq:mu-module}, with $k=\ell=g$ and $x=1$, then gives
\[
 c=\tau_g(1,y+1)\tau_g(-1,y)=1.
\]
Thus the nonsplit extension is obstructed already at the cochain level.
For the split extension, $\Psi=\tau=\mu=1$ supplies the transitive rank-one
datum. This example shows why the extension coherence cannot be discarded
even when the action on the character lattice is fixed. \qed
\end{example}

\section{Commutative direct products and mixed cocycles}
\label{sec:commutative-direct-products}

Assume in this section that $G:=T\times K$ and that $K$ is finite {\bf abelian}.  Then $G$
is diagonalizable and
$\mathscr{O}(G)=\mathbb{C}[\widehat T\times \widehat K]$.
Hopf $2$-cocycles on $\mathscr{O}(G)$ are therefore ordinary normalized group $2$-cocycles on the
abelian group $\widehat T\times \widehat K$ \cite{G1}.

\subsection{The canonical K\"unneth decomposition}

\begin{theorem}\label{thm:kunneth}
There is a canonical group isomorphism
\begin{equation}\label{eq:kunneth}
 H^2(\widehat T\times \widehat K,\Ctimes)
 \cong
 H^2(\widehat T,\Ctimes)
 \times \Hom(\widehat T\otimes_{\mathbb Z}\widehat K,\Ctimes)
 \times H^2(\widehat K,\Ctimes).
\end{equation}
If $\psi$ and $\xi$ are normalized $2$-cocycles on $\widehat T$ and $\widehat K$, respectively, and
$\chi:\widehat T\times \widehat K\to\Ctimes$ is a bicharacter, then the corresponding class is
represented by
\begin{equation}\label{eq:mixed-cocycle}
 J_{\psi,\chi,\xi}\bigl((x,\alpha),(y,\beta)\bigr)
 =\psi(x,y)\chi(x,\beta)\xi(\alpha,\beta).
\end{equation}
The inverse of \eqref{eq:kunneth} sends $[J]$ to
\[
 \left(
 [J|_{(\widehat T\times\{1\})^2}],
 \chi_J,
 [J|_{(\{1\}\times \widehat K)^2}]
 \right),
 \qquad
 \chi_J(x,\beta):=R^J((x,1),(1,\beta)).
\]
\end{theorem}

\begin{proof}
Because $\Ctimes$ is divisible, it is an injective $\mathbb Z$-module.  The
universal coefficient sequence therefore identifies
$H^2(A,\Ctimes)\cong\Hom(\textstyle\bigwedge^2 A,\Ctimes)$
for every abelian group $A$; under this identification, a cocycle class is sent to
its alternating bicharacter $R^J:=J_{21}^{-1}J$.  The canonical decomposition
\[
 \bigwedge^2(\widehat T\times \widehat K)
 \cong
 \bigwedge^2\widehat T\oplus(\widehat T\otimes_{\mathbb Z}\widehat K)\oplus\bigwedge^2\widehat K
\]
then gives \eqref{eq:kunneth}.  Formula \eqref{eq:mixed-cocycle} is a normalized
$2$-cocycle, and its alternating bicharacter has the prescribed three components.
This also proves the formula for the inverse.
\end{proof}

Equivalently, a bicharacter $\chi$ is the same as a homomorphism
\[
 \lambda:\widehat T\longrightarrow K=\widehat {\widehat K},
 \qquad
 \chi(x,\beta)=\beta(\lambda(x)).
\]
In these terms,
\begin{equation}\label{eq:R-block}
 R^J((x,\alpha),(y,\beta))
 =R^\psi(x,y)\,
 \beta(\lambda(x))\,
 \alpha(\lambda(y))^{-1}\,
 R^\xi(\alpha,\beta).
\end{equation}

\subsection{Minimality}

\begin{proposition}\label{prop:radical}
For the cocycle in \eqref{eq:mixed-cocycle}, written by means of $\lambda$, one has
\[
 \Rad(R^J)=
 \left\{
 (x,\alpha)\in \widehat T\times \widehat K\ \middle|\
 \begin{array}{l}
 R^\psi(x,-)=\alpha\circ\lambda,\\[1mm]
 R^\xi(\alpha,-)=\operatorname{ev}_{\lambda(x)}^{-1}
 \end{array}
 \right\}.
\]
Thus, $J$ is minimal if and only if this set is trivial.
\end{proposition}

\begin{proof}
Fix $(x,\alpha)\in \widehat T\times \widehat K$. Then setting $\beta=1$ in
\eqref{eq:R-block} shows that 
$$R^J((x,\alpha),(y,1))=1$$
for every $y$ exactly when
$R^\psi(x,-)=\alpha\circ\lambda$.  Setting $y=1$ in
\eqref{eq:R-block} shows that
$$R^J((x,\alpha),(1,\beta))=1$$
for every $\beta$ exactly when
$R^\xi(\alpha,-)=\operatorname{ev}_{\lambda(x)}^{-1}$.  Since the two factors
generate $\widehat T\times \widehat K$, these conditions are necessary and sufficient.
\end{proof}

\subsection{Comparison with the equivariantization classification}
\label{subsec:kunneth-versus-equivariantization}

The classification in Theorem~\ref{thm:kunneth} and the classification in
Corollary~\ref{cor:hopf-cocycle-classification} describe the same
gauge-equivalence classes, but they use rather different coordinates.  The
K\"unneth coordinates are canonical and cohomological, whereas the data of
\S\ref{sec:fiber-functors} are orbit--stabilizer data arising from
equivariantization.  We now make the relation between them explicit.

Let
\[
 [J]=[J_{\psi,\chi,\xi}]
 \in H^2(\widehat T\times \widehat K,\Ctimes)
\]
be written in the K\"unneth form of Theorem~\ref{thm:kunneth}, and let
\[
 \lambda:\widehat T\longrightarrow K=\widehat{\widehat K},
 \qquad
 \chi(x,\beta)=\beta(\lambda(x)),
\]
be the homomorphism corresponding to the mixed bicharacter.  Define
homomorphisms
\[
 r_\psi:\widehat T\longrightarrow T,
 \qquad
 y(r_\psi(x))=R^\psi(x,y),
\]
and
\[
 r_\xi:\widehat K\longrightarrow K,
 \qquad
 \beta(r_\xi(\alpha))=R^\xi(\alpha,\beta).
\]
Also let
\[
 \lambda^\vee:\widehat K\longrightarrow T,
 \qquad
 x(\lambda^\vee(\alpha))=\alpha(\lambda(x)),
\]
be the homomorphism dual to $\lambda$.  Let
\[
 T_\psi:=\overline{\operatorname{im}(r_\psi)}
       =\Rad(R^\psi)^\perp\subseteq T,
 \qquad
 N_\xi:=\operatorname{im}(r_\xi)
       =\Rad(R^\xi)^\perp\subseteq K.
\]
Thus, $T_\psi$ and $N_\xi$ are the supports of the two diagonal K\"unneth
components.

We have the following dictionary between the two classifications.

\begin{proposition}
\label{prop:kunneth-equivariantization-dictionary}
Let $\mathcal D_J$ be the equivalence class of rank-one admissible data in
\S\ref{sec:fiber-functors} corresponding to $J$.  Then one may choose a
representative
\[
 \mathcal D_J=(S_J,\Psi_J,L_J,U_J,\tau_J,\mu_J)
\]
with the following properties:
\begin{enumerate}

\item The subgroup and orbit data are
\begin{equation}\label{eq:dictionary-S-L}
 N:=N_\xi,
 \qquad
 L_J:=\lambda(\widehat T)N,
 \qquad
 S_J:=\lambda^{-1}(N).
\end{equation}
There is a canonical isomorphism
\begin{equation}\label{eq:dictionary-orbit}
 \widehat T/S_J\xrightarrow{\cong}L_J/N,
 \qquad
 xS_J\longmapsto\lambda(x)N,
\end{equation}
and, under this identification, $U_J$ is the left translation action of
$L_J$ on $L_J/N$.  In particular, the stabilizer of the base point is $N$.

\item The restriction class $[\xi]\in H^2(\widehat K,\Ctimes)$ determines the
nondegenerate stabilizer class
\[
 [\mu_J]\in H^2(N,\Ctimes).
\]
More precisely, if $B:=\Rad(R^\xi)$, then
$\widehat K/B\cong\widehat N$, and the alternating form induced by $R^\xi$ on
$\widehat K/B$ is nondegenerate. Choose a character of the center
$\mathbb C[B]$ of $\mathbb C^\xi[\widehat K]$, after trivializing $\xi$ on
$B$, and take the corresponding simple matrix block of
$\mathbb C^\xi[\widehat K]$. The stabilizer $N$ acts projectively on its
unique simple module. The multiplier is $\mu_J$, and it is nondegenerate.  The resulting cohomology class $[\mu_J]$ depends only on
$[\xi]$, up to the harmless inversion dictated by the convention for
projective representations.

The cochains $\Psi_J$ and $\tau_J$ encode the toral and mixed components
relative to this finite block; hence they can also involve $\xi$ through the
chosen projective intertwiners.  They depend on choices of cocycle
representatives and intertwiners, but their equivalence class does not.  In
particular, the isotropy cocycle
\[
 \psi_{\mathcal D_J}(s,t)=\Psi_J(s,t;S_J)
\]
need not coincide with the restriction
$\psi|_{S_J\times S_J}$; it is the Mackey obstruction obtained after the
mixed and finite data have been incorporated.

\item The K\"unneth coordinates are recovered from the cocycle associated
with $\mathcal D_J$ by
\begin{equation}\label{eq:recover-kunneth}
 \left[
 J\big|_{(\widehat T\times\{1\})^2}
 \right]=[\psi],
 \quad
 R^J((x,1),(1,\beta))=\chi(x,\beta),
 \quad
 \left[
 J\big|_{(\{1\}\times \widehat K)^2}
 \right]=[\xi].
\end{equation}
Thus, two rank-one admissible data give the same K\"unneth triple if
and only if they are equivalent in the sense of
\S\ref{sec:fiber-functors}.

\item Let $H_J\subseteq T\times K$ be the support of $J$, and let
$\Lambda_{\mathcal D_J}$ be the isotropy radical from
Theorem~\ref{prop:factorization}.  Then
\begin{equation}\label{eq:dictionary-support}
 \pi(H_J)=L_J=\lambda(\widehat T)N_\xi
\end{equation}
and
\begin{equation}\label{eq:dictionary-isotropy-radical}
 \Lambda_{\mathcal D_J}
 =\operatorname{pr}_{\widehat T}\Rad(R^J)
 =
 \left\{
 x\in \widehat T\ \middle|\
 \begin{array}{l}
 \text{there exists }\alpha\in \widehat K\text{ such that}\\[1mm]
 R^\psi(x,-)=\alpha\circ\lambda,\,
 R^\xi(\alpha,-)=\operatorname{ev}_{\lambda(x)}^{-1}
 \end{array}
 \right\}.
\end{equation}

\end{enumerate}
\end{proposition}

\begin{proof}
Consider first the twisted group algebra $\mathbb C^\xi[\widehat K]$, and let
$B:=\Rad(R^\xi)$.
Since the restriction of $\xi$ to $B$ has trivial alternating bicharacter,
it is a coboundary; after replacing $\xi$ by a cohomologous cocycle, we may
therefore assume that this restriction is trivial.  Then
$Z(\mathbb C^\xi[\widehat K])=\mathbb C[B]$,
and the simple matrix blocks of $\mathbb C^\xi[\widehat K]$ are indexed by
$\widehat{B}\cong K/N_\xi$.

For $x\in \widehat T$, tensoring a projective $\widehat K$-module by the character
$\chi(x,-)$ translates its block by $\lambda(x)N_\xi$.  Hence the orbit of
the base block is
$\lambda(\widehat T)N_\xi/N_\xi=L_J/N_\xi$,
and its stabilizer in $\widehat T$ is $\lambda^{-1}(N_\xi)$.  This gives
\eqref{eq:dictionary-S-L}, \eqref{eq:dictionary-orbit}, and the translation
action $U_J$.

Let $e$ be the primitive central idempotent selecting the base block, and
write $u_\alpha$ for the standard basis of
$\mathbb C^\xi[\widehat K]$. For $n\in N_\xi$, choose
$\alpha_n\in\widehat K$ with $r_\xi(\alpha_n)=n$. The identity
\[
 u_{\alpha_n}u_\beta u_{\alpha_n}^{-1}
 =R^\xi(\alpha_n,\beta)u_\beta=\beta(n)u_\beta
\]
shows that conjugation by $u_{\alpha_n}$ implements the action of $n$ on
the block. The induced map $\widehat K/B\to N_\xi$ is an isomorphism,
so the elements $e u_{\alpha_n}$ form a basis of
$e\mathbb C^\xi[\widehat K]$. They are invertible in this block and
multiply projectively. Choosing this multiplier convention defines
$\mu_J$ and gives an algebra isomorphism
\[
 \mathbb C^{\mu_J}[N_\xi]
 \xrightarrow{\ \cong\ }e\mathbb C^\xi[\widehat K]
 \cong\End(V),
\]
where $V$ is the simple module of the block. In particular, the projective
action generates the whole block, $(\dim V)^2=|N_\xi|$, and $\mu_J$ is
nondegenerate. Changing the lifts or the intertwiners changes the multiplier
by a coboundary. The $K$-action is transitive on the blocks and commutes
with the $N_\xi$-action, so changing the base block does not change the
resulting cohomology class. Compatible intertwiners also produce $\Psi_J$ and
$\tau_J$, and their choices change only the representative of the
admissible data. Formula~\eqref{eq:recover-kunneth} is exactly the inverse
map in Theorem~\ref{thm:kunneth}.

For the support statement, define
\begin{equation}\label{eq:rho-J}
 \rho_J:\widehat T\times \widehat K\longrightarrow T\times K,
 \qquad
 \rho_J(x,\alpha)
 =
 \bigl(
 r_\psi(x)\lambda^\vee(\alpha)^{-1},
 \lambda(x)r_\xi(\alpha)
 \bigr).
\end{equation}
Using \eqref{eq:R-block}, one obtains
\[
 R^J((x,\alpha),(y,\beta))
 =(y,\beta)\bigl(\rho_J(x,\alpha)\bigr).
\]
Therefore
$\Rad(R^J)=\ker(\rho_J)$ and $H_J=\overline{\operatorname{im}(\rho_J)}$.
Taking the second component in \eqref{eq:rho-J} gives
$\pi(H_J)
 =
 \lambda(\widehat T)\operatorname{im}(r_\xi)
 =
 \lambda(\widehat T)N_\xi$,
which proves \eqref{eq:dictionary-support}.

Since
$H_J^\perp=\Rad(R^J)$,
elementary character duality gives
\[
 (H_J\cap T)^\perp=\operatorname{pr}_{\widehat T}\Rad(R^J).
\]
Indeed, let $x\in\widehat T$ be trivial on $H_J\cap T$. Restrict
$x\circ\operatorname{pr}_T$ to $H_J$. This character descends to
$H_J/(H_J\cap T)\subseteq K$. Extend its inverse to a character
$\alpha\in\widehat K$. Then $(x,\alpha)$ is trivial on $H_J$, so
$x\in\operatorname{pr}_{\widehat T}H_J^\perp$. The reverse inclusion
follows by restricting to $H_J\cap T$.
Theorem~\ref{prop:factorization} now gives
$\Lambda_{\mathcal D_J}
 =(H_J\cap T)^\perp
 =\operatorname{pr}_{\widehat T}\Rad(R^J)$.
The explicit description in
\eqref{eq:dictionary-isotropy-radical} follows from
Proposition~\ref{prop:radical}.
\end{proof}

The preceding dictionary makes the comparison of the two minimality criteria
completely transparent.

\begin{corollary}
\label{cor:kunneth-equivariantization-minimality}
With the notation above, the following conditions are equivalent:
\begin{enumerate}

\item $J$ is minimal.

\item
$
 \Rad(R^J)=\{1\}$.

\item
$
 \Rad(R^\psi)=\{1\}$ and 
$
 \ker(\lambda^\vee)\cap\Rad(R^\xi)=\{1\}$.

\item
$
 T_\psi=T$ and $K=\lambda(\widehat T)N_\xi$.

\item the \S\ref{sec:fiber-functors} datum satisfies
$
 L_J=K$ and $\Lambda_{\mathcal D_J}=\{1\}$.
\end{enumerate}
Thus condition~(5) is precisely the minimality criterion of
Theorem~\ref{prop:factorization}, while conditions~(2)--(4) are its
K\"unneth formulations.
\end{corollary}

\begin{proof}
The equivalence of (1) and (2) is the definition of minimality for a
diagonalizable group.  

If (2) holds, then testing the radical on $(x,1)$ and
$(1,\alpha)$ gives
\[
 \ker(\lambda)\cap\Rad(R^\psi)=\{1\},
 \qquad
 \ker(\lambda^\vee)\cap\Rad(R^\xi)=\{1\}.
\]
Since $\lambda(\widehat T)$ is finite, $\ker(\lambda)$ has finite index in the
torsion-free group $\widehat T$.  Hence every nontrivial subgroup of $\widehat T$ meets
$\ker(\lambda)$ nontrivially, so
$ker(\lambda)\cap\Rad(R^\psi)=\{1\}$ if and only if 
$\Rad(R^\psi)=\{1\}$.
This proves $(2)\Rightarrow(3)$.

Conversely, assume (3), and let
$(x,\alpha)\in\Rad(R^J)$.
If $m$ is the order of $\alpha$, then
$(x^m,1)\in\Rad(R^J)$
because $R^J$ is a bicharacter.  Hence
$^m\in\Rad(R^\psi)=\{1\}$.
Since $\widehat T$ is torsion-free, $x=1$.  It follows that
$\alpha\in
 \ker(\lambda^\vee)\cap\Rad(R^\xi)
 =\{1\}$,
so (2) holds.

By annihilator duality,
$\Rad(R^\psi)^\perp=T_\psi$,
$\ker(\lambda^\vee)^\perp=\lambda(\widehat T)$,
and $\Rad(R^\xi)^\perp=N_\xi$.
Therefore (3) and (4) are equivalent.  

Finally,
by \eqref{eq:dictionary-support} $L_J=K$ if and only if $\ker(\lambda^\vee)\cap\Rad(R^\xi)=\{1\}$.
Moreover, by \eqref{eq:dictionary-isotropy-radical},
$\Lambda_{\mathcal D_J}=\{1\}$ if and only if $\Rad(R^\psi)=\{1\}$.
Indeed,
$$\Rad(R^\psi)\cap\ker(\lambda)
 \subseteq
 \operatorname{pr}_{\widehat T}\Rad(R^J).$$
Conversely, if
$x\in\operatorname{pr}_{\widehat T}\Rad(R^J)$,
choose $\alpha$ with $(x,\alpha)\in\Rad(R^J)$ and raise to the order of
$\alpha$.  A power of $x$ then lies in
$\Rad(R^\psi)\cap\ker(\lambda)$.
Since $\widehat T$ is torsion-free and $\ker(\lambda)$ has finite index, the claimed
equivalence follows, proving the equivalence with (5).
\end{proof}

\begin{remark}
\label{rem:mixed-component-comparison}
The subgroup $N_\xi$ (and not $L_J$) is the subgroup carrying the
nondegenerate finite stabilizer cocycle $\mu_J$.  The mixed component enlarges
it to
$L_J=\lambda(\widehat T)N_\xi$,
and the quotient
$L_J/N_\xi\cong \widehat T/S_J$
is exactly the transitive orbit appearing in the rank-one criterion of
\S\ref{sec:fiber-functors}.  Thus, even for a minimal Hopf $2$-cocycle, one may
have
$N_\xi\subsetneq K$
and a degenerate restriction $\xi$ on $\widehat K$; the missing finite directions are
supplied by $\lambda(\widehat T)$.

On the torus side the situation is different: since $\widehat T$ is torsion-free and
$\widehat K$ is finite, a nonzero element of $\Rad(R^\psi)$ always has a nonzero power
in $\ker(\lambda)$ and therefore gives a nontrivial element of
$\Rad(R^J)$.  Hence minimality forces $R^\psi$ itself to be nondegenerate.
The mixed component can compensate for degeneracy of the finite block, but
not for degeneracy of the toral block.

If the mixed component is trivial, then $\lambda=1$, so
\[
 S_J=\widehat T,
 \qquad
 L_J=N_\xi,
 \qquad
 U_J=\id.
\]
In this special case minimality does reduce to separate nondegeneracy of the
two diagonal blocks.  By contrast, the image of the mixed component in
$K/N_\xi$ is manifested in \S\ref{sec:fiber-functors} by the quotient
$\widehat T/S_J$ and its transitive $L_J$-action.  The part of $\lambda(\widehat T)$ lying
inside $N_\xi$ does not enlarge the orbit; instead, it is absorbed into the
cochains $\Psi_J$ and $\tau_J$ and can change the isotropy cocycle even when
$S_J=\widehat T$. \qed
\end{remark}

\begin{remark}
The two classifications have complementary advantages.  The K\"unneth
classification is canonical, carries the abelian group structure on
$H^2(\widehat T\times \widehat K,\Ctimes)$ componentwise, and displays the mixed term explicitly.  The
equivariantization classification is less canonical at the cochain level,
but it records the orbit, stabilizer, support, and factorization of the
associated module category, and it continues to make sense for nonsplit
extensions and for noncommutative component groups.

Notice also that the cochain $\tau_J$ in the admissible datum is not, in
general, a canonical homomorphism $K\to T$.  The canonical mixed invariant is
$\lambda:\widehat T\to K$,
or equivalently its dual
$\lambda^\vee:\widehat K\to T$.
Converting this into a map $K\to T$ would require a noncanonical
identification $K\cong\widehat K$.  In the commutative direct-product case,
the identities above show that the two classifications agree exactly. \qed
\end{remark}

\begin{example}[A purely mixed class]\label{ex:purely-mixed}
Let $G:=\Gm\times \mathbb{Z}/2\mathbb{Z}$, so $\widehat T=\mathbb Z$ and $\widehat K=\mathbb{Z}/2\mathbb{Z}$.  Writing elements as
$(m,\varepsilon)$ with $\varepsilon\in\{0,1\}$, define
\[
 J((m,\varepsilon),(n,\delta))=(-1)^{m\delta}.
\]
Both restrictions of $J$ to the two factors are trivial, but
\[
 R^J((1,0),(0,1))=-1.
\]
Thus $[J]$ is the nontrivial mixed class in
$\Hom(\mathbb Z\otimes (\mathbb{Z}/2\mathbb{Z}),\Ctimes)$.  It is invisible on both diagonal factors and is detected entirely by the mixed component. \qed
\end{example}

\begin{example}[A minimal cocycle with degenerate finite block]
Let
$$
 G:=(\Gm)^2\times (\mathbb{Z}/2\mathbb{Z}),$$
 so that $\widehat T=\mathbb Z^2$,
$\widehat K=\mathbb{Z}/2\mathbb{Z}$, 
and choose $q\in\Ctimes$ that is not a root of unity.  Define
\[
 J\bigl(((a,b),\varepsilon),((c,d),\delta)\bigr)
 =q^{ad}(-1)^{a\delta}.
\]
Then
\[
 R^J\bigl(((a,b),\varepsilon),((c,d),\delta)\bigr)
 =q^{ad-cb}(-1)^{a\delta+c\varepsilon}.
\]
If $((a,b),\varepsilon)$ lies in the radical, testing against $((0,1),0)$ gives
$q^a=1$, hence $a=0$.  Testing against $((1,0),0)$ then gives
$q^{-b}(-1)^\varepsilon=1$, hence $b=0$ and $\varepsilon=0$.  Therefore $R^J$ is
nondegenerate, although the restriction of $J$ to $\mathbb{Z}/2\mathbb{Z}\times \mathbb{Z}/2\mathbb{Z}$ is trivial.
This illustrates why minimality must be tested on the full block bicharacter. \qed
\end{example}

\subsection{An explicit isotropy-radical formula}
\label{subsec:explicit-isotropy-radical}

The dictionary also allows to compute the isotropy radical without
first constructing the cochains $\Psi_J$ and $\tau_J$. Let
\[
 B:=\Rad(R^\xi),\qquad N:=\operatorname{im}(r_\xi),\qquad
 S:=\lambda^{-1}(N).
\]
For $s\in S$, choose $\alpha_s\in\widehat K$ with
$r_\xi(\alpha_s)=\lambda(s)$, and define
\begin{equation}\label{eq:explicit-isotropy-form}
 b_S(s,t):=R^\psi(s,t)\,R^\xi(\alpha_s,\alpha_t)^{-1},
 \qquad s,t\in S.
\end{equation}

\begin{proposition}\label{prop:explicit-isotropy-radical}
The expression in \eqref{eq:explicit-isotropy-form} is independent of the
choices of $\alpha_s$ and is an alternating bicharacter on $S$. Moreover,
$\Rad(b_S)=\operatorname{pr}_{\widehat T}\Rad(R^J)
          =\Lambda_{\mathcal D_J}$.
Thus the isotropy radical is determined directly by the three K\"unneth
components.
\end{proposition}

\begin{proof}
The class of $\alpha_s$ modulo $B$ is uniquely determined and depends
homomorphically on $s$. Since $R^\xi$ descends to an alternating
bicharacter on $\widehat K/B$, the first claim follows.

Suppose $(s,\alpha)\in\Rad(R^J)$. By Proposition~\ref{prop:radical},
$r_\xi(\alpha)=\lambda(s)^{-1}$. Hence $s\in S$ and
$\alpha=\alpha_s^{-1}\beta$ for some $\beta\in B$. Such a $\beta$ is
trivial on $N$, so, for $t\in S$, the other radical condition gives
\[
 R^\psi(s,t)=\alpha(\lambda(t))
 =\alpha_s(\lambda(t))^{-1}=R^\xi(\alpha_s,\alpha_t).
\]
Thus $s\in\Rad(b_S)$.

Conversely, suppose $s\in\Rad(b_S)$. The character
\[
 d(y)=R^\psi(s,y)\,\alpha_s(\lambda(y)),\qquad y\in\widehat T,
\]
is trivial on $S$. Under the isomorphism
$\widehat T/S\cong\lambda(\widehat T)N/N$, it is a character of a
subgroup of $K/N$. Extend it to a character of $K/N$, and regard this
extension as $\beta\in B\subseteq\widehat K$. Then
$\alpha=\alpha_s^{-1}\beta$ satisfies
$R^\psi(s,-)=\alpha\circ\lambda$ and $r_\xi(\alpha)=\lambda(s)^{-1}$.
These are the radical conditions in Proposition~\ref{prop:radical}, so
$(s,\alpha)\in\Rad(R^J)$. The last equality follows from
Proposition~\ref{prop:kunneth-equivariantization-dictionary}(4).
\end{proof}

\section{Partial results on Question~\ref{q:torus}}
\label{sec:torus-question}

Let $G$ be an affine algebraic group over $\mathbb C$. We first recall invariant
Hopf $2$-cocycles, and then prove Question~\ref{q:torus} for minimal Hopf $2$-cocycles which are 
deformations of the trivial Hopf $2$-cocycle.

\subsection{Invariant representatives}

A Hopf $2$-cocycle $J$ on $\Os(G)$ is \textbf{invariant} if
$m*J=J*m$, where $m$ is multiplication on $\Os(G)$.

\begin{corollary}[{\cite[Theorem 4.2]{EGinv}}]
\label{prop:invariant-minimal}
Let $G$ be an affine algebraic group over $\mathbb C$. If a minimal Hopf $2$-cocycle on
$\Os(G)$ is gauge equivalent to an invariant Hopf $2$-cocycle, then $G$ is
commutative. In particular, if $G$ is reductive, then $G^\circ$ is a torus.
\end{corollary}

\begin{proof}
Minimality is invariant under gauge equivalence. The support of an
invariant Hopf $2$-cocycle is a commutative closed normal subgroup by
\cite[Theorem 4.2]{EGinv}. Minimality makes the support equal to $G$.
A connected commutative reductive group is a torus.
\end{proof}

Thus a counterexample to Question~\ref{q:torus} would have to involve a
minimal gauge class with no invariant representative.

\subsection{Regular families, scalar extension, and locality}
\label{subsec:regular-families}

Let $\Sigma$ be an irreducible normal affine algebraic curve over $\mathbb C$ with marked points
$0,a$. A \textbf{regular family} is a normalized $\Os(\Sigma)$-valued Hopf $2$-cocycle
\begin{equation}\label{nfht}
 J:\Os(G)^{\otimes2}\longrightarrow \Os(\Sigma),\qquad J(0)=\varepsilon\otimes\varepsilon
\end{equation}
with an $\Os(\Sigma)$-valued convolution inverse \cite[Section 6]{EG1}. Set
$R^J:=J_{21}^{-1}*J$.

Choose a local parameter
$h$ at $0$. Then expansion at $0$ gives an embedding $\mathbb C(\Sigma)\hookrightarrow \mathbb C((h))$.
Let
\[
 \Os(G)_E:=\Os(G)\otimes_{\mathbb C}E,\qquad E:=\mathbb C((h)).
\]

\begin{lemma}
\label{lem:generic-minimality}
Let $A$ be an integral $\mathbb C$-algebra, $V$ a $\mathbb C$-vector space,
and $B:V\times V\to A$ a bilinear form. Suppose that for a homomorphism
$a:A\to\mathbb C$ the specialization $B_a$ is nondegenerate. Then the
extended form is nondegenerate on $V\otimes_{\mathbb C}E'$ for every field
extension $E'/\operatorname{Frac}(A)$. 

In particular, if $J(a)$ is minimal, then the $R$-form is nondegenerate on $\Os(G)_E$.
\end{lemma}

\begin{proof}
Given linearly independent $v_1,\ldots,v_n\in V$, the linear functionals
$B_a(v_i,-)$ are independent. Hence there are $w_1,\ldots,w_n\in V$ such
that
\[
 a\!\left(\det\bigl(B(v_i,w_j)\bigr)_{i,j=1}^n\right)\neq0.
\]
This determinant is a nonzero element of $A$, so it is invertible in $E'$.
Thus, no nonzero $E'$-linear combination of the $v_i$ lies in the
left radical of the extended form. Every element of
$V\otimes_{\mathbb C}E'$ is such a finite linear combination. Interchanging
the two arguments proves the claim for the right radical.
\end{proof}

\begin{lemma}
\label{lem:formal-gauge-coefficients}
Let $b\in \Os(G)^*[[h]]$ be a normalized convolution-invertible
element with $b=\varepsilon+O(h)$. The map
\[
 T_b(f)=b(f_1)f_2b^{-1}(f_3)
\]
is an invertible $E$-linear map of $\Os(G)_E$. With the gauge convention of
\S\ref{sec:hopf-cocycles}, one has
\begin{equation}\label{eq:formal-gauge-R}
 R^{J^b}(f,g)=R^J(T_bf,T_bg).
\end{equation}
Thus formal gauge transformations preserve nondegeneracy on $\Os(G)_E$.
\end{lemma}

\begin{proof}
Every $f\in \Os(G)$ belongs to a finite-dimensional subcoalgebra $C\subseteq \Os(G)$.
The formula gives $T_b(C)\subseteq C[[h]]$, and
$C[[h]]=C\otimes_{\mathbb C}\mathbb C[[h]]$ since $C$ is finite
dimensional. Hence $T_b$ maps $\Os(G)_E$ to itself. Its inverse is $T_{b^{-1}}$,
which has the same property.

Since $\Os(G)$ is commutative, we have
$R^{J^b}=(b\otimes b)*R^J*(b^{-1}\otimes b^{-1})$.
Evaluating this identity gives \eqref{eq:formal-gauge-R}, and invertibility
of $T_b$ proves the last claim.
\end{proof}

\begin{proposition}
\label{prop:deformation-connectedness}
If a regular family $J$ (\ref{nfht}) has a minimal fiber, then $G$ is connected.
\end{proposition}

\begin{proof}
By Lemma~\ref{lem:generic-minimality}, the formal expansion of $R^J$ is
nondegenerate on $\Os(G)_E$. The locality lemma
\cite[Lemma 6.5]{EG1} provides a formal gauge transformation after which
$J$ is a local twist
$\mathcal F\in\bigl(U(\mathfrak g)^{\otimes2}\bigr)[[h]]$ with $\mathcal F=1+O(h)$, where 
$\mathfrak g:=\operatorname{Lie}(G^\circ)$. By Lemma~\ref{lem:formal-gauge-coefficients}, its $R$-form remains
nondegenerate on $\Os(G)_E$.

Now if $G$ were disconnected, choose a nonzero regular function $f$ supported on a
connected component other than $G^\circ$ (such a function exists since the
components are open and closed). Since every distribution supported at the
identity annihilates $f$, and since all coefficients of the local $R$-matrix
belong to $U(\mathfrak g)^{\otimes2}$, $f$ belongs to the left radical
of the local form, which is a contradiction.
\end{proof}

\subsection{The formal normal form and semisimple coefficients}
\label{subsec:semisimple-coefficients}

We use the triangular quantization normal form from
\cite{EG1}, based on the
quantization theory of \cite{EK1,EK2}. Let $\Delta_0$ be the ordinary coproduct and for a 
twist $\mathcal F$ of $U(\mathfrak g)[[h]]$, let
\begin{equation}\label{locf}
 \Delta^{\mathcal F}(u):=\mathcal F^{-1}\Delta_0(u)\mathcal F,
 \qquad \mathcal R^{\mathcal F}:=\mathcal F_{21}^{-1}\mathcal F.
\end{equation}

\begin{proposition}
\label{prop:triangular-normal-form}
Let $\mathfrak g$ be a finite-dimensional $\mathbb C$-Lie algebra and let
$\mathcal F=1+O(h)$ be a local formal twist of $U(\mathfrak g)[[h]]$. The following hold:
\begin{enumerate}
\item
There exist a solution
$r(h)\in h\bigl(\textstyle\bigwedge^2\mathfrak g\bigr)[[h]]$
of the classical Yang--Baxter equation, a universal quantization twist
$\mathcal F_r$, and a continuous $\mathbb C[[h]]$-linear triangular Hopf
isomorphism
\begin{equation}\label{eq:triangular-normal-form}
 \phi:U(\mathfrak g)[[h]]^{\mathcal F}
       \xrightarrow{\ \cong\ }
       U(\mathfrak g)[[h]]^{\mathcal F_r},
 \qquad \phi\equiv\id\pmod h.
\end{equation}
\item
In particular,
\begin{equation}\label{eq:normal-form-R}
 \mathcal R_r=(\phi\otimes\phi)(\mathcal R^{\mathcal F}),
 \qquad \mathcal R_r:=\mathcal F_{r,21}^{-1}\mathcal F_r.
\end{equation}
The elements $\mathcal F_r$ and $\mathcal R_r$ are convergent formal
series of homogeneous universal expressions in $r(h)$.
\item
Let
$\mathfrak q:=\operatorname{im}\bigl(r(h)^\sharp:
                      \mathfrak g_E^*\longrightarrow\mathfrak g_E\bigr)$.
Each such homogeneous expression for $\mathcal R_r$ belongs to
$U_E(\mathfrak q)\otimes_E U_E(\mathfrak q)$.
\end{enumerate}
\end{proposition}

\begin{proof}
(1)--(2) The triangular normal form is the one used in \cite[proof of Theorem 6.4, following Lemma 6.5]{EG1}. 
The universal construction leaves the product of $U(\mathfrak g)[[h]]$ undeformed and expresses its twist in terms of the classical
$r$-matrix; see \cite[Section 1.4]{EK2}. We have absorbed the quantization
parameter into $r(h)$, which explains the factor $h$ in its stated range.
The isomorphism in the triangular category gives
\eqref{eq:normal-form-R}. If its reduction modulo $h$ is a Lie algebra
automorphism $\alpha$ rather than $1$, compose in the target with
$U(\alpha^{-1})$ and replace $r(h)$ by
$(\alpha^{-1}\otimes\alpha^{-1})r(h)$. Functoriality of the universal
construction then gives \eqref{eq:triangular-normal-form} with the stated
normalization.

(3) The universal expressions consist of products of tensor
components of $r(h)$. Since over $E=\mathbb C((h))$, those components belong to the carrier
$\mathfrak q$, each homogeneous expression has the stated support.
Since $r(h)$ is divisible by $h$, these expressions tend to zero
$h$-adically as their degree tends to infinity.
\end{proof}

For a connected affine algebraic group $H$ over $\mathbb{C}$, regard $\Os(H)$ as its
space of rational matrix coefficients in $U(\operatorname{Lie}H)^*$.
The pairing with distributions at the identity is injective in
characteristic zero. After scalar extension to $E$, a function in $\Os(H)_E$
defines a continuous $\mathbb C[[h]]$-linear map
$U(\operatorname{Lie}H)[[h]]\to E$ whose values have a uniform finite
pole order.

\begin{lemma}
\label{lem:central-primitive-subspace}
Suppose $\mathfrak g=\mathfrak s\oplus\mathfrak z$ is reductive, where
$\mathfrak s:=[\mathfrak g,\mathfrak g]$ and $\mathfrak z$ is the center of $\mathfrak g$.
The isomorphism $\phi$ in \eqref{eq:triangular-normal-form} preserves
$\mathfrak z[[h]]$ and the closed two-sided ideal it generates. Also, $\phi$ induces
an algebra automorphism
\[
 \bar\phi:U(\mathfrak s)[[h]]\longrightarrow U(\mathfrak s)[[h]],
 \qquad \bar\phi\equiv\id\pmod h.
\]
\end{lemma}

\begin{proof}
Every $z\in\mathfrak z$ is central, so $z\otimes1+1\otimes z$ commutes with every twist. 
Thus $z$ is primitive for both 
$\Delta^{\mathcal F}$ and $\Delta^{\mathcal F_r}$.
Therefore $\phi(z)$ is central and primitive in $U(\mathfrak g)[[h]]^{\mathcal F_r}$. Since the element
$\phi(z)\otimes1+1\otimes\phi(z)$ is central in the tensor product, untwisting the coproduct shows that 
$\phi(z)$ is also primitive for
$\Delta_0$. Since the central primitive elements of $U(\mathfrak g)[[h]]$ are
exactly $\mathfrak z[[h]]$, coefficient by coefficient, $\phi(z)\in \mathfrak z[[h]]$. Applying the same
argument to $\phi^{-1}$ gives
$\phi(\mathfrak z[[h]])=\mathfrak z[[h]]$. Passing to the quotient by the
closed ideal generated by this subspace proves the claim.
\end{proof}

\begin{lemma}
\label{lem:semisimple-matrix-rigidity}
Let $H$ be a connected semisimple affine algebraic group over $\mathbb C$, with
Lie algebra $\mathfrak s$. If $\psi$ is a continuous
$\mathbb C[[h]]$-algebra automorphism of $U(\mathfrak s)[[h]]$ with
$\psi\equiv\id\pmod h$, then pullback by $\psi$ and by $\psi^{-1}$
preserves $\Os(H)_E$ inside the continuous dual of
$U(\mathfrak s)[[h]]$.
\end{lemma}

\begin{proof}
Let $\rho:U(\mathfrak s)\to\End_{\mathbb C}(V)$ be the representation
associated with a finite-dimensional rational $H$-module $V$ and extend it to
$U(\mathfrak s)[[h]]$ by $\mathbb C[[h]]$-linearity. Consider the
representation $\rho\circ\psi$ on $V[[h]]$; it is a formal deformation of
$\rho$. Using the first Whitehead lemma
$H^1(\mathfrak s,\End_{\mathbb C}(V))=0$,
one can use induction to construct successive changes of basis that conjugate $\rho\circ\psi$ to $\rho$. 
Namely suppose $\rho\circ\psi$ and $\rho$ have been made
equal modulo $h^n$. Their discrepancy in degree $n$, restricted to
$\mathfrak s$, is a Lie-algebra $1$-cocycle, hence a coboundary, and
conjugating by a matrix of the form $I+h^n B_n$ removes it. Since the $h$-adic
product of these conjugations converges, there is
$B\in\operatorname{GL}(V[[h]])$, with $B\equiv I\pmod h$, such that
$\rho\circ\psi=B\rho B^{-1}$ on all of $U(\mathfrak s)[[h]]$.

Therefore pullback by $\psi$ sends each matrix coefficient of $\rho$ to
a $\mathbb C[[h]]$-linear combination of the finitely many coefficients of
$\rho$. Since these lie in their algebraic scalar extension to $E=\mathbb C((h))$, and every
regular function on $H$ is a matrix coefficient of a finite-dimensional
rational representation, it follows that $\psi^*$ preserves $\Os(H)_E$. The same argument works for $\psi^{-1}$.
\end{proof}

\subsection{Carriers and the deformation theorem}

Recall that the
\textbf{algebraic hull} of a Lie subalgebra is the smallest connected
algebraic subgroup whose Lie algebra contains it.

\begin{lemma}
\label{lem:carrier-algebraic-hull}
Let $\mathfrak p$ be a Lie subalgebra of the Lie algebra of an affine 
algebraic group over $\mathbb C$. If $P$ is its
connected algebraic hull and $\mathfrak a=\operatorname{Lie}(P)$, then
$[\mathfrak a,\mathfrak a]\subseteq\mathfrak p$.
\end{lemma}

\begin{proof}
The stabilizer of $\mathfrak p$ under the adjoint action is a closed
algebraic subgroup whose Lie algebra contains $\mathfrak p$. Minimality
of the hull implies that $P$ stabilizes $\mathfrak p$. Thus
$\mathfrak p$ is an ideal of $\mathfrak a$. Consider the induced
representation of $P$ on $\mathfrak a/\mathfrak p$. Its differential
vanishes on $\mathfrak p$, since
$[\mathfrak p,\mathfrak a]\subseteq\mathfrak p$. Since the kernel is a
closed algebraic subgroup whose Lie algebra contains $\mathfrak p$,
minimality of the hull again implies that the representation of $P$ is
trivial. Therefore its differential is zero on all of $\mathfrak a$,
which is the claim.
\end{proof}

\begin{lemma}[{\cite[Theorem 6.6]{EG1}}]
\label{lem:reductive-carrier}
Let $\mathfrak q$ be a finite-dimensional reductive Lie algebra over $\mathbb C$. If $\mathfrak q$ is the carrier of a skew
solution of the classical Yang--Baxter equation, then $\mathfrak q$ is
abelian.
\end{lemma}

\begin{proof}
The bivector is nondegenerate on its carrier, and its inverse is a
nondegenerate Lie-algebra $2$-cocycle $\omega$. Write
$\mathfrak q=\mathfrak s\oplus\mathfrak c$, with $\mathfrak s$
semisimple and $\mathfrak c$ central. The cocycle identity and
$[\mathfrak s,\mathfrak s]=\mathfrak s$ imply
$\omega(\mathfrak s,\mathfrak c)=0$. By the second Whitehead lemma,
\[
 \omega(x,y)=\ell([x,y])\qquad(x,y\in\mathfrak s)
\]
for some $\ell\in\mathfrak s^*$. Write $\ell(x)=\kappa(v,x)$ using
the Killing form. If $v\neq0$, then
$\omega(v,x)=\kappa(v,[v,x])=0$ for every $x\in\mathfrak s$, so $v$
is in the radical of $\omega$. If $v=0$ and $\mathfrak s\neq0$, 
$\mathfrak s$ is contained in the radical. Both contradict
nondegeneracy. Thus $\mathfrak s=0$.
\end{proof}

We can now state and prove the main result of this section.

\begin{theorem}
\label{thm:deformation-torus}
Let $G$ be a reductive affine algebraic group over $\mathbb C$. Suppose
that a regular family of Hopf $2$-cocycles on $\Os(G)$ (\ref{nfht}) has a trivial
fiber and a minimal fiber. Then $G$ is a torus.
\end{theorem}

\begin{proof}
Use the notation of \S\ref{subsec:regular-families}.
By Proposition~\ref{prop:deformation-connectedness}, $G$ is connected.
After the formal gauge transformation, the local form
$\mathcal R^{\mathcal F}$ (\ref{locf}) is nondegenerate on $\Os(G)_E$ by
Lemmas~\ref{lem:generic-minimality} and
\ref{lem:formal-gauge-coefficients}. Write
\[
 \mathfrak g=\mathfrak s\oplus\mathfrak z,\qquad
 G_{\mathrm{ss}}:=G/Z(G)^\circ.
\]
Then $G_{\mathrm{ss}}$ is connected semisimple with Lie algebra
$\mathfrak s$, and the quotient map identifies
$\Os(G_{\mathrm{ss}})_E$ with a subspace of $\Os(G)_E$.

Apply Proposition~\ref{prop:triangular-normal-form}, and let
\[
 \mathfrak q:=\operatorname{im}r(h)^\sharp,\qquad
 \mathfrak p:=\operatorname{pr}_{\mathfrak s_E}(\mathfrak q).
\]
The classical Yang--Baxter equation makes $\mathfrak q$ a Lie subalgebra
of $\mathfrak g_E$ on which $r(h)$ is nondegenerate. Let
$P\subseteq(G_{\mathrm{ss}})_E$ be the connected algebraic hull of
$\mathfrak p$.

We claim that $P=(G_{\mathrm{ss}})_E$. Suppose $P\neq(G_{\mathrm{ss}})_E$, and choose
\[
 0\neq f\in I(P)\subseteq\Os(G_{\mathrm{ss}})_E.
\]
As a functional on $U_E(\mathfrak g)$, the pullback of $f$ factors
through $U_E(\mathfrak s)$. It annihilates $U_E(\mathfrak q)$ since the
projection of that algebra is contained in $U_E(\mathfrak p)$, and all
these distributions annihilate the ideal of $P$. By the support property
of the universal expressions in
Proposition~\ref{prop:triangular-normal-form}, it follows that
\begin{equation}\label{eq:universal-first-slot-annihilation}
 (f\otimes\lambda)(\mathcal R_r)=0
\end{equation}
for every continuous $\mathbb C[[h]]$-linear functional $\lambda$ on
$U(\mathfrak g)[[h]]$ whose values have a uniform finite pole order.
Indeed, $f$ annihilates the first tensor factor in each homogeneous
expression, and evaluation commutes with their $h$-adic sum.

Now by Lemma~\ref{lem:central-primitive-subspace}, $\phi$ descends to
$\bar\phi$ on $U(\mathfrak s)[[h]]$. Therefore Lemma~\ref{lem:semisimple-matrix-rigidity}
gives
\[
 f':=f\circ\phi\in\Os(G_{\mathrm{ss}})_E\subseteq\Os(G)_E,
 \qquad f'\neq0.
\]
Here $f$ and $f'$ are viewed as functionals on $U(\mathfrak g)[[h]]$
through the quotient. For any $g\in\Os(G)_E$, the functional
$g\circ\phi^{-1}$ is continuous and has a uniform finite pole order. 
Taking $\lambda:=g\circ\phi^{-1}$ in
\eqref{eq:universal-first-slot-annihilation} and using
\eqref{eq:normal-form-R}, we obtain
\begin{align*}
 \mathcal R^{\mathcal F}(f',g)
 &=\bigl(f\otimes(g\circ\phi^{-1})\bigr)
                \bigl((\phi\otimes\phi)(\mathcal R^{\mathcal F})\bigr)\\
 &=\bigl(f\otimes(g\circ\phi^{-1})\bigr)(\mathcal R_r)=0.
\end{align*}
All evaluations are defined $h$-adically since algebraic coefficients over $E$
have a finite pole order, $\phi^{\pm1}$ preserve the $h$-adic filtration,
and the universal homogeneous expressions tend to zero. Thus the
nonzero algebraic coefficient $f'$ lies in the radical of
$\mathcal R^{\mathcal F}$, which is a contradiction. Thus
$P=(G_{\mathrm{ss}})_E$, as claimed.

Finally, Lemma~\ref{lem:carrier-algebraic-hull}, applied to
$\mathfrak p\subseteq\mathfrak s_E$, now gives
$\mathfrak s_E=[\mathfrak s_E,\mathfrak s_E]
 \subseteq\mathfrak p$,
so $\mathfrak p=\mathfrak s_E$. Since $\mathfrak z_E$ is central,
$[\mathfrak q,\mathfrak q]\subseteq\mathfrak s_E$, and projection to
$\mathfrak s_E$ sends this commutator algebra onto
$[\mathfrak p,\mathfrak p]=\mathfrak s_E$. Since the projection is the
identity on $\mathfrak s_E$, $\mathfrak s_E\subseteq\mathfrak q$.
Thus,
$\mathfrak q=\mathfrak s_E\oplus(\mathfrak q\cap\mathfrak z_E)$
is reductive. By Lemma~\ref{lem:reductive-carrier}, $\mathfrak q$
is abelian, and therefore $\mathfrak s=0$. Since $G$ is connected and
reductive, it is a torus.
\end{proof}

\begin{remark}
We stress that assuming the Hopf $2$-cocycle occurs in a regular deformation of the
trivial Hopf $2$-cocycle is essential to the proof. \qed
\end{remark}

\enlargethispage{3\baselineskip}

\end{document}